%% file: main.tex
\documentclass[10pt]{article}
\usepackage[utf8]{inputenc}
\usepackage[style=numeric, maxbibnames=99, doi=false,isbn=false,url=false,eprint=false]{biblatex}
\usepackage{graphicx}
\usepackage{todonotes} 
\usepackage{bm}
\usepackage{amsmath}
\usepackage{mathtools}
\usepackage{amsfonts}
\usepackage{amssymb}
\usepackage{amsthm}

\usepackage{cleveref}
\usepackage{enumitem}
\usepackage{caption}
\usepackage{subcaption}
\graphicspath{ {images/} }

\theoremstyle{remark}

\theoremstyle{plain}
\newtheorem{theorem}{Theorem}[section]
\newtheorem*{theorem*}{Theorem}
\newtheorem{corollary}[theorem]{Corollary}
\newtheorem{lemma}[theorem]{Lemma}
\theoremstyle{definition}
\newtheorem{remark}[theorem]{Remark}
\newtheorem*{remark*}{Remark}
\newtheorem{definition}[theorem]{Definition}
\newtheorem{assumption}[theorem]{Assumption}

\input{notation}

\bibliography{main}

\title{First Order Logic of Sparse Graphs with Given Degree Sequences}
\author{Alberto Larrauri, Guillem Perarnau}
\date{May 2023}

\begin{document}

\maketitle

\input{sections/abstract}

\input{sections/intro}

\input{sections/preliminaries}

\input{sections/outline}

\input{sections/probability}

\input{sections/fragment}

\input{sections/supercritical}
\input{sections/subcritical}

\input{sections/convergence_law}

\printbibliography

\input{sections/appendix}

\end{document}

%% file: notation.tex
\newcommand{\NN}{\mathbb{N}}

\newcommand{\FF}{\textsc{Fr}}
\newcommand{\TT}{\textsc{Fo}}

\newcommand{\ZZpos}{{\NN}}
\newcommand{\Ln}{\lim_{n\to \infty}}

\newcommand{\bd}{{\bm{d}}}
\newcommand{\Conf}{\mathbb{CM}}
\newcommand{\Gcal}{\mathbb{G}}

\newcommand{\Hcal}{\mathcal{H}}

\newcommand{\Ucal}{\mathcal{U}}

\renewcommand{\Pr}{\mathbb{P}}

\newcommand{\E}[1]{\mathbb{E}\left[ #1
\right] }
\newcommand{\Var}[1]{\mathrm{Var}\left( 
#1 \right)}

\DeclareMathOperator{\fog}{\textsc{FO}}
\DeclareMathOperator{\fo}{\textsc{FO}}
\DeclareMathOperator{\Core}{\mathrm{Core}}

\DeclareMathOperator{\ex}{\mathrm{ex}}
\DeclareMathOperator{\aut}{\mathrm{aut}}
\DeclareMathOperator{\authe}{{\mathrm{aut}_{\mathrm{h.e.}}}}

\newcommand{\mmid}{\; \middle| \;}

\newcommand{\mmnt}{\rho}

\newcommand{\Frag}{\mathrm{Frag}}

\newcommand{\acyc}{\mathrm{acyc}}

\newcommand {\defin}[1] {{\it{#1}}}
\newcommand{\comment}[1]{}

%% file: sections/abstract.tex
\begin{abstract}
    We consider limit probabilities of first order properties in random graphs with a given degree sequence. Under mild conditions on the degree sequence, we show that the closure set of limit probabilities is a finite union of closed intervals. Moreover, we characterize the degree sequences for which this closure set is the interval $[0,1]$, a property that is intimately related with the probability that the random graph is acyclic. As a side result, we compile a full description of the cycle distribution of random graphs and study their fragment (disjoint union of unicyclic components) in the subcritical regime. Finally, we amend the proof of the existence of limit probabilities for first order properties in random graphs with a given degree sequence; this result was already claimed by Lynch~[IEEE LICS 2003] but his proof contained some inaccuracies.
\end{abstract}

%% file: sections/intro.tex
\section{Introduction}
\label{sec:intro}

Since the seminal work of  Erd\H os and R\'enyi~\cite{ER60}, random graphs have been central objects of
study in probabilistic combinatorics. In this area it is common to ask, given a graph property $\varphi$ and
a sequence of random graphs $\Gcal_n$ increasing in order, what is the limit probability that $\Gcal_n$ 
satisfies $\varphi$. Of course, this limit may or may not exist, and the question about its existence is interesting on its own. In a related direction, for families $\mathcal{L}$ of ``well-behaved'' graph properties, one may want to obtain a procedure (i.e., an algorithm) that given some property $\varphi$ in $\mathcal{L}$
computes the \textit{limit probability} $p(\varphi):=\Ln \Pr(\Gcal_n \text{ satisfies } \varphi)$, if it exists. \par
The \defin{model-theoretical} approach to the questions above is to classify graph properties according to the formal languages that can express them, and obtain convergence results for the entire language rather than for individual properties. One such language is the first order ({\sc FO}) language of graphs, which consists of first order logic, where variables represent vertices, plus a binary adjacency relation, which is meant to be symmetric and anti-reflexive. Fagin~\cite{faginProbabilitiesFiniteModels1976} and, independently, Glebski et. al.~\cite{glebskiiRangeDegreeRealizability1969} showed that $p(\varphi)\in \{0,1\}$ for all {\sc FO}-properties $\varphi$ in the case that $\Gcal_n=\Gcal_n(1/2)$, the \emph{binomial random graph} on $n$ vertices, obtained by including each edge independently with probability $1/2$.
Even more, there is a procedure to determine the limit probability for any given {\sc FO}-property. This result, which gave rise to the model-theoretical study of random graphs, is an example of a \defin{zero-one law}. A more general kind of result is a \defin{convergence law}, which simply states that for a given sequence of random graphs $\Gcal_n$ and a given language $\mathcal{L}$, the limit probability $p(\varphi)$ exists for all properties $\varphi$ expressible in $\mathcal{L}$.
This work continues a line of research~\cite{hmnt2018,larrauriLimitingProbabilitiesFirst2022} that studies the geometry of the set of limit probabilities in settings where a convergence law holds. \par

In this paper we deal with the case where $\Gcal_n$ is a random graph whose degree sequence
has been fixed a priori. A \defin{degree sequence} of length $n$ is a sequence $\bd_n=
(d_i)_{i\in [n]}$ where $d_i$ is a non-negative integer with $d_i< n$ for all $i\in [n]$ and 
$\sum_{i\in [n]} d_i$ is even.
We call $\bd_n$ \defin{feasible}, if there is at least one graph $G$ with $V(G)=[n]$
whose degree sequence is $\bd_n$, meaning that for all $v\in [n]$, it has degree $d_v$.
Given a feasible degree sequence $\bd_n$ on $n$ vertices, $\Gcal_n(\bd_n)$ denotes the uniform random graph with vertex set $[n]$ and whose degree sequence is $\bd_n$.  
A \defin{sequence of degree sequences} is $\bd=(\bd_n)_{n\in \NN}$, where
$\bd_n=(d_{n,i})_{i\in [n]}$ is a degree sequence on $n$ vertices. By convention, we set $d_{n,i}=0$ for each $i\geq n$, and, if the context is clear, we use $d_i=d_{n,i}$. 
Given $\bd$, we define the \defin{sequence of random graphs} $\Gcal(\bd) = (\Gcal_n(\bd_n))_{n\in \NN}$. 
\par

In order to study $\fo$ logic on $\Gcal(\bd)$, we will need to impose some regularity conditions on $\bd$.
These conditions are better stated in terms of degree distributions. For $k\geq 0$, define $n_k=n_{n,k}= |\{i\in [n] \mid  d_{n,i}=k\}|$. Given $n\in \NN$ and $\bd$, the \defin{degree distribution} $D_n=D_n(\bd)$ is given by $\Pr(D_n = k) = n_{n,k}/n$. Equivalently, $D_n$
is the probability distribution of the degree of a uniform random vertex in $\Gcal_n(\bd_n)$. The following is our main assumption, used throughout this paper.

\begin{assumption}
\label{assump:main}
There exists a probability distribution $D=D(\bd)$ on $\NN_0$  such that

\begin{enumerate}[label=(\Roman*)]
    \item[\defin{(i)}]\label{item:assump_1}  $\bd_n$ is feasible for all $n\in \NN$;
    \item[\defin{(ii)}]\label{item:assump_2} 
    $ D_n \rightarrow D$, in distribution;
    \item[\defin{(iii)}]\label{item:assump_3}  $\lim_{n\to \infty} 
    \E{D_n}$ (resp.,
    $\lim_{n\to \infty}
    \E{D_n^2}$),
    exists, is bounded and 
    equals to $\E{D}$ (resp., $\E{D^2}$);
    \item[\defin{(iv)}]\label{item:assump_4}  if $\Pr(D=k)=0$ for some $k\in \NN_0$, then $n_{n,k}=0$ for all $n\in \NN$.
\end{enumerate}
\end{assumption}

In this context a convergence law for $\fo$-properties holds.

\begin{theorem}
\label{thm:convergence_law}
Suppose that $\bd$ satisfies \Cref{assump:main}. Then for any
property $\varphi$ in the $\fo$ language of graphs, the following limit exists
\begin{equation}
\label{eq:limits}
  p(\varphi,\bd):= \Ln \Pr(\Gcal_n(\bd_n) \text{ satisfies } \varphi).
\end{equation}
\end{theorem}

The proof of this theorem is not the primary goal of this paper and is sketched in \Cref{sec:convergence_law}. Very similar results were established by Lynch in two closely related articles~\cite{lynchConvergenceLawRandom2005,lynchConvergenceLawRandom2003}. The difference between Lynch's articles is the assumptions imposed on $\bd$. These assumptions are non-comparable with \Cref{assump:main}. However, due to a slight oversight, both Lynch's proofs are incorrect. Even more, the convergence law does not always exist under any of his two assumptions. Nevertheless, under~\Cref{assump:main} his proof strategy can be used correctly and this assumption is essentially the weakest condition required to have a convergence law. 
\par

Given that the limit in \labelcref{eq:limits} exists for all $\fog$-properties $\varphi$, our object of interest is the set of limits. Define
\begin{equation}
\label{eq:L_bd}
L(\bd) := \{
p(\varphi,\bd)  \, \mid \, \varphi \text{ $\fog$-property}
\}.
\end{equation}

 We are interested in the geometry of the \defin{topological closure} $\overline{L(\bd)}$ (i.e. the union of the points in the set and its limit points). 
 Observe that $\overline{L(\bd)}$ is a \defin{symmetric} subset of the interval $[0,1]$ (i.e, $p\in \overline{L(\bd)}$ if
and only if $1-p \in \overline{L(\bd)}$), since the negation of an $\fog$-property is also an $\fog$-property.

 The main result of the paper is the following.

\begin{theorem}
\label{thm:main}
Suppose that $\bd$ satisfies \Cref{assump:main}. Define
\begin{equation}
\label{eq:p_acyc}
p_{\acyc}(\bd):=\Ln \Pr(\Gcal_n(\bd_n) \text{ is acyclic} ). 
\end{equation}
Then, $\overline{L(\bd)}$ is a finite union of closed intervals, and
\begin{enumerate}
	\item[(1)] if $p_{\acyc}(\bd) < 1/2$, then $\overline{L(\bd)}\neq [0,1]$;
	\item[(2)] if $p_{\acyc}(\bd) \geq 1/2$, then 
 $\overline{L(\bd)} = [0,1]$.
\end{enumerate}
\end{theorem}

We devote the rest of the introduction to discuss some aspects of our results.
\begin{remark}[Discussion on \Cref{assump:main}]\label{rem:1}
 Assumption (i) is necessary in order to define $\Gcal(\bd)$. Assumptions (ii) and (iii) allow us to study $\Gcal_n(\bd_n)$ by looking at the limit degree distribution $D$. They imply that the average degree and the second moment of the degree sequence are bounded in probability, which in particular imply that the maximum degree is $o(\sqrt{n})$. These two assumptions are usual in the setting of random graphs with given degree sequence. The infinite degree variance case exhibits a very different behaviour (see e.g. \cite{vanderhofstadRandomGraphsComplexVolII}) and it remains as an open problem to prove if in such case $\Gcal(\bd)$ has a convergence law for the $\fog$ language. Finally, Assumption (iv) rules out the existence of vertices with low-frequency degrees. Otherwise such vertices would pose an obstacle to a convergence law for $\fo$ logic on graphs: for instance, consider $\bd_n$ containing a single vertex of degree $3$ for odd $n$, and none for even $n$, and $\varphi$ the {\sc FO}-property ``the graph contains a vertex of degree $3$'', then $p(\varphi,\bd)$ does not exist. For our purposes, assumption (iv) could be weakened replacing ``for all $n$'' by ``for all sufficiently large $n$''. However, we use the stronger version for convenience. \par
 \end{remark}

 \begin{remark}[The configuration model]\label{rem:2} As it is usually the case, instead of studying directly the graph $\Gcal_n(\bd_n)$ we will study $\Conf_n(\bd_n)$, a related random (multi)graph known as the \defin{configuration model}, and introduced in \Cref{sec:prelims}. Theorems~\ref{thm:convergence_law} and~\ref{thm:main} also hold for $\Conf(\bd)$, as it will be discussed later.
 \end{remark}

 \begin{remark}[Probability of being acyclic]\label{rem:3}
The cycle distribution and the limit probability that $\Gcal_n(\bd_n)$ (or $\Conf_n(\bd_n)$) is acyclic have been studied in the literature (see \Cref{sec:previous}). It is well-known that, provided that $\bd$ satisfies \Cref{assump:main}, the number of cycles of length $k$ converges to a Poisson variable with parameter $\nu^k/(2k)$, and they are asymptotically independent (see \Cref{thm:cycle_distribution}). Here, $\nu$ is a limit parameter defined on $\bd$, see~\eqref{eq:def_nu}, which in fact is a fundamental parameter of random graphs with a given degree sequence. For instance, the phase transition for the existence of a giant component is located at $\nu=1$~\cite{janson2009new,molloyCriticalPointRandom1995}. As a consequence, $p_{\acyc}(\bd)$ exists and only depends on $\nu$:
\begin{equation}
    \label{eq:p_acyc_expression}
    p_{\acyc}(\bd) = \sqrt{1-\nu}\cdot e^{\frac{\nu}{2} + \frac{\nu^2}{4}} \qquad \text{for }\nu\in (0,1),
\end{equation}
and $p_{\acyc}(\bd)=0$ if $\nu \geq 1$ (see \Cref{cor:acyclic}). In \eqref{eq:p_acyc_expression}, the term $\sqrt{1-\nu}$ accounts for the probability that $\Conf_n(\bd_n)$ is acyclic: it has no cycles of length $k$, for all $k\geq 1$. However, our object of interest is $\Gcal_n(\bd_n)$ which, by definition, is simple: it has neither cycles of length one (loops) nor cycles of length two (multiedges). The correction terms $e^{{\nu}/{2}}$ and $e^{{\nu^2}/{4}}$ can be thought as deducting the contribution of loops and multiedges, respectively, from  the configuration model acylic probability.

Alternatively, taking the Taylor series of $\ln(\sqrt{1-\nu})$, we can rewrite~\eqref{eq:p_acyc_expression} as
\begin{equation}\label{eq:alter_acyclic}
p_{\acyc}(\bd) = \exp\left(-\sum_{k\geq 3} \frac{\nu^k}{2k} \right) \qquad \text{for }\nu\in (0,1).
\end{equation}
This is not a closed-form expression but it has a more straightforward interpretation: $\Gcal_n(\bd_n)$ is acyclic if and only if it has no cycles of length three (term $e^{-\nu^3/6}$), no cycles of length four (term $e^{-\nu^4/8}$), and so on. In particular, note that $p_{\acyc}(\bd)\in (0,1)$ for $\nu\in (0,1)$.
 
The second part of \Cref{thm:main} locates a threshold at $p_{\acyc}(\bd)=1/2$. Let $\nu_0\in [0,1]$ be the unique root of
\begin{equation}\label{eq:nu0}
    \sqrt{1-\nu}\cdot e^{\frac{\nu}{2} + \frac{\nu^2}{4}}=\frac{1}{2}.
\end{equation}
Indeed, the previous equation has exactly one solution in $[0,1]$ as its {\sc LHS} is monotonically decreasing in this interval and evaluates to $1$ at $\nu=0$ and to $0$ at $\nu=1$. Numerically, we obtain $\nu_0\approx 0.9368317$, which coincides with the threshold value for $c$ in $\Gcal_{n}(c/n)$ obtained in~\cite{larrauriLimitingProbabilitiesFirst2022}. To our best knowledge, the value $\nu_0$ has not been identified yet as a threshold for any other property of $\Gcal_n(\bd_n)$.
\end{remark}

\par

%% file: sections/preliminaries.tex
\section{Preliminaries}\label{sec:prelims}

\subsection{Notation}

From now on we fix $\bd$ that satisfies \Cref{assump:main}. 
We introduce some additional notation that will be used throughout the paper. 
For $k\geq 0$ and $n\in \NN$, define 
$\lambda_{n,k}:= \Pr(D_n=k)$ and $\lambda_k := \Pr(D=k)$, where 
$D_n$ is the degree distribution corresponding to $\bd_n$ and
$D$ is the limiting degree distribution (which exists by \Cref{assump:main}).
Additionally, define $\widehat{\lambda}_n:=1 - \lambda_{n,0} - \lambda_{n,1}$,
and $\widehat{\lambda}:=\Ln \widehat{\lambda}_n$. 
Similarly, define $\widehat{n}:= |\left\{ i\in [n] \mid d_{n,i}\geq 2 \right\}|=n \widehat{\lambda}_n$, the number of vertices whose degree in $\bd_n$ is at least $2$. 

Define $m_n:= \sum_{i\in [n]}d_{n,i}$,
twice the number of edges in $\Gcal_n(\bd_n)$.
Define $\mmnt_{n,k}:=\E{D_n(D_n-1)\dots (D_n-k+1)}$ and
$\mmnt_k:=\E{D(D-1)\dots (D-k+1)}$, the $k$-th factorial moments of $D_n$ and $D$, respectively. By \Cref{assump:main}.(iii), $\Ln\mmnt_{n,k}= \mmnt_k$ for $k=1,2$.
Define 
\begin{equation}
\label{eq:def_nu}
    \nu:= \nu(\bd)=\frac{\E{D(D-1)}}{\E{D}}= \frac{\rho_2}{\rho_1}.
\end{equation}
By the convergence of the first and second moments, if $\nu_n :=\tfrac{\E{D_n(D_n-1)}}{\E{D_n}}$, then $\lim_{n\to\infty}\nu_n= \nu$. \par

As we already indicated in \Cref{rem:3}, $\nu$ plays a crucial role in the cycle distribution and the shape of the limit probability set.

\subsection{Multigraphs}

A \textit{multigraph} $G$ is a pair $(V(G), E(G))$ where $V(G)$ is its \textit{vertex set}, and $E(G)$ is its \textit{edge set}, a multiset of unordered pairs $\{u,v\}$ where $u,v\in V(G)$. We allow the possibility that $u=v$, in which case the edge is called a \textit{loop}. Given $u,v\in V(G)$, the \textit{multiplicity} of the edge with endpoints $u,v$ is the number of pairs $\{u, v\}$ in the multiset $E(G)$. The \textit{degree} $\deg(v)$ of a vertex $v\in V(G)$ is the number of edges $\{u,v\}\in E(G)$ with $u\neq v$, plus twice the number of loops $\{v,v\}\in E(G)$.

If $G,H$ are multigraphs, an $H$-copy in $G$ is a \emph{sub-multigraph} $H^\prime \subseteq G$ that is isomorphic to $G$.
In the context of multigraphs, $k$-\defin{cycles} are defined as usual for $k\geq 3$. For $k=2$, a $2$-cycle consists of two vertices plus two edges joining them, and for $k=1$, a $1$-cycle is just a vertex with a loop attached to it.
\par
Define the \defin{excess} 
of a multigraph (or graph) $G$ to be $\ex(G):=|E(G)|-|V(G)|$. A connected graph $G$ is \defin{unicylic} if $\ex(G)=0$, or, equivalently, when it has exactly one cycle. A \defin{fragment} is a disjoint union of unicylic components.
The \defin{fragment} $\Frag(G)$ of a multigraph (or graph) $G$ is the union of its unicyclic components.
We call a fragment \defin{simple} if it contains neither loops nor multiedges, or, in other words, if it contains no cycles of length smaller than $3$. 
\par

Given a multigraph $G$, its number of \defin{half-edge automorphisms} is 
\begin{equation}
\label{eq:def_authe}
    \authe(G):=\aut(G) 2^\ell \prod_{u,v\in V(G)}m({u,v})!,
\end{equation}
where $\ell$ is the number of loops in $G$, and $m(u,v)$ denotes the multiplicity of $\{u,v\}$ in $E(G)$. In this way, $\authe(C_k)=2k$, where $C_k$ is a $k$-cycle and $k\geq 1$. Informally, $\authe(G)$ is the number of half-edge permutations in $G$ that preserve both incidence to the same vertex and the matching between half-edges.  

Other notions related to  graphs extend to multigraphs in the natural way.   \par

\subsection{Configuration Model}
We study random graphs with given degree sequence through the so-called \defin{configuration model}, introduced by Bollobás \cite{benderAsymptoticNumberLabeled1978, bollobas1980probabilistic}, which instead yields a random multigraph with the desired degree sequence. See e.g. \cite{vanderhofstadRandomGraphsComplexVolI,vanderhofstadRandomGraphsComplexVolII} for the basic properties of the model.

Given a degree sequence $\bd_n$ of length $n$, the \defin{configuration model $\Conf_n(\bd_n)$} is a uniform random matching of $[m_n]$  (formally, $\Conf_n(\bd_n)\subseteq \binom{[m_n]}{2}$), where we recall that
$m_n = \sum_{i\in [n]} d_i$. 
We refer to
the elements $e\in [m_n]$ as \defin{half-edges}. We say that a half-edge $e\in [m_n]$ is \defin{incident} to a vertex $v\in [n]$ if $\sum_{u < v} d_u <
e \leq \sum_{u\leq v} d_u$. In other words, the first $d_1$ half-edges belong to vertex $1$, the following $d_2$ belong to vertex $2$, and so on. The \defin{underlying multigraph}
of $\Conf_n(\bd_n)$ has vertex set $[n]$ and the number of edges between two different vertices $v_1, v_2\in [n]$ is the number of pairs $\{h_1,h_2\}$ in the matching $\Conf_n(\bd_n)$
where $h_1,h_2$ are incident to $v_1,v_2$, respectively. 

In the following, we identify $\Conf_n(\bd_n)$ with its underlying multigraph. 
Informally, we obtain the multigraph $\Conf_n(\bd_n)$ by attaching $d_v$ half edges to each vertex $v\in [n]$, and matching these half-edges randomly afterwards.
For a sequence of degree sequences $\bd=(\bd_n)_{n\in \NN}$, we denote by $\Conf(\bd)= (\Conf_n(\bd_n))_{n\in \NN}$.
\par

The probability that $\Conf_n(\bd_n)=G$ for a fixed multigraph $G$ with degree sequence $\bd_n$ depends only on its number of loops and the multiplicity of its edges. In particular, if $\bd$ is feasible, conditioning $\Conf_n(\bd_n)$ on being simple (i.e. the absence of loops and multiedges) results in $\Gcal_n(\bd_n)$. It is thus natural to compute the probability that $\Conf_n(\bd_n)$ is simple. The following theorem gives an answer for sequences that satisfy our assumption (in fact, the theorem holds in a slightly more general setting). Recall the definition of $\nu$ in \eqref{eq:def_nu}.
\begin{theorem}[Bollobás \cite{bollobas1980probabilistic}; Janson{{~\cite[Theorem 1.1]{janson2009probability}}}]\label{thm:simple}
Let $\bd$ satisfies \Cref{assump:main}. Then 
\begin{align}\label{eq:simple}
\liminf_{n\to\infty} \Pr(\Conf_n(\bd_n) \text{ is simple})= e^{-\frac{\nu}{2}-\frac{\nu^2}{4}} > 0.
\end{align}
\end{theorem}
\noindent We will revisit this theorem in forthcoming sections.

Let us also briefly recall the role of $\nu$ in the phase transition of $\Conf(\bd)$. Under \Cref{assump:main}, $\nu$ determines the appearance of a giant component in $\Conf(\bd)$ \cite{janson2009new,molloyCriticalPointRandom1995,}. Namely,  $\Conf_n(\bd_n)$ a.a.s. contains a linear order component if and only if $\nu>1$. \par

\subsection{Exchange of Limit and Sum Operators} 
Some of our results rely on the exchange of limit and sum operations. We recall the notion of tight sequence and an equivalent characterization.
\begin{definition}
\label{def:uniform_sum}
Let $(f_n)_{n\in \NN}$ be a sequence of functions $f_n: S\rightarrow
[0,\infty)$ where $S$ is a countable set. The sequence 
$(f_n)_{n\in \NN}$ is \defin{tight} if for every $\epsilon>0$ there exists a finite $T\subset S$ satisfying 
$\sum_{s\not\in T} f_n(s) < \epsilon$ for all $n\in \NN$.
\end{definition}

\begin{lemma}
\label{lem:tight}
Let $(f_n)_{n\in \NN}$ be a sequence of functions $f_n: S\rightarrow
[0,\infty)$ where $S$ is a countable set. Suppose that 
\begin{itemize}
    \item [(1)] for each $s\in S$, 
$f(s)=\lim_{n\to\infty} f_n(s)$ exists and is finite,
    \item [(2)] $\sum_{s\in S} f(s)$ is finite.
\end{itemize}
Then $\lim_{n\to\infty}
\sum_{s\in S} f_n(s)
= \sum_{s\in S} f(s)$
    if and only if
    $(f_n)_{n\in \NN}$ is tight. 
\end{lemma}

\subsection{Probability preliminaries}

 In order to study the small cycle distribution of $\Conf(\bd)$ we will need the next auxiliary result~\cite[Theorem 1.23]{bollobas2001random}.
\begin{theorem}[Method of Moments for Poisson random variables
]
	\label{thm:moments_method}
	Let $k\in \NN$. For each $n\in \NN$, let $X_{n,1},$ $\dots, X_{n,k}$ be random variables over the
	same measurable space. Suppose that there are real
	positive constants $\lambda_1,\dots, \lambda_k$
	such that for all $a_1,\dots, a_k\in \ZZpos$ 
	\[
	\lim_{n\to \infty}
	\E{
		\prod_{i=1}^k \binom{X_{n,i}}{a_i}
	} = \prod_{i=1}^k \frac{\lambda_i^{a_i}}{a_i!}.
	\]
	Then $(X_{n,1},\dots, X_{n,k})$ converges in distribution to a vector of independent Poisson random variables with parameters $\lambda_1,\dots,\lambda_k$ as $n\to\infty$.
\end{theorem}

\subsection{Sets of Partial Sums}

Given a convergent series $\sum_{n\in \NN} p_n$ whose terms are non-negative, its \defin{set of partial sums} is defined as $\{ 
\sum_{n\in A} p_n \colon A \subseteq \NN \}\subseteq [0,+\infty)$. As part of the proof of \Cref{thm:main} we need to analyze the geometry of some of these sets.
Our tool for this matter is the following classical result conjectured by Kakeya \cite{kakeya} and later proven in \cite{nymann2000paper}. 

\begin{lemma}[Kakeya's Criterion]
	\label{ch:binom_lims.kakeya}
	Let $\sum _{n\in \NN}p_n$ be a convergent series of non-negative real numbers such that $p_n\geq p_{n+1}$ for all $n\in \NN$. Then 
	the following are equivalent:
	\begin{itemize}
		\item[(1)] $p_i \le \sum_{j>i} p_j$ for all $i \in \NN$.
		\item[(2)]
		\[
		\left\{\sum_{
n\in A} p_n  \colon  A \subseteq  \mathbb{N}\right\}
		= \left[0, \sum_{n\in\NN} p_n\right].
		\]
	\end{itemize}
	If the condition (1) only holds for all values of $i$ 
	large enough, 
	then the set $\left\{\sum_{n \in A} p_n  \colon  A \subseteq  \mathbb{N}\right \}$
	is a finite union of intervals. 
\end{lemma}

%% file: sections/outline.tex
\section{Outline of the proof of \texorpdfstring{\Cref{thm:main}}.}

Here we briefly outline the proof of our main result.
Throughout the proof, we will suppose that $\bd$ satisfies \Cref{assump:main}. For the sake of conciseness, this will be implicitly assumed in all our statements.

The first step towards the result is to understand the distribution of short cycle in $\Gcal_n(\bd_n)$, which is done in \Cref{sec:cycles}. This section contains some already known results that are dispersed in the literature and a secondary goal is to collect them all in a single document. The most important conclusion of this part is \Cref{cor:acyclic}, which states that the asymptotic probability that $\Gcal_n(\bd_n)$ is acyclic admits a nice expression in terms of the parameter $\nu$:
\[
p_{\acyc}(\bd)=
\sqrt{1 - \nu } \cdot e^{\frac{\nu}{2} + \frac{\nu^2}{4}}.
\]
The study of the cycle structure gives access to the distribution of the unicyclic components of $\Gcal_n(\bd_n)$, which form the fragment of the random graph. This is done in \Cref{sec:fragment}. We obtain an expression for the probability a particular fixed graph is the fragment of $\Gcal_n(\bd_n)$ (\Cref{cor:prob_frag}), and show that its size distribution is tight, provided that $\nu<1$ (\Cref{cor:frag_integrable}).

From here on, our strategy is similar to the one used in \cite{larrauriLimitingProbabilitiesFirst2022} to deal with the binomial random graph $\Gcal_n(c/n)$, however, the probabilistic arguments here are more convoluted. 

We split the proof of \Cref{thm:main} into two parts. In the first part, developed in \Cref{sec:subcritical}, we show that $\overline{L(\bd)}$ consists of a finite union of closed intervals when $0< \nu < 1$ (\Cref{lem:finite_gaps}).  \Cref{thm:fragment_zeroone} establishes that the subcritical regime the random graph fragment determines whether a given {\sc FO}-property is satisfied or not. This is used to prove that $\overline{L(\bd)}$ is the set of partial sums of fragment probabilities (\Cref{thm:partial_sums}). The desired result is obtained by Kakeya's Criterion. Compared to the approach in \cite{larrauriLimitingProbabilitiesFirst2022}, fragment probabilities have a more complex structure that heavily depends on the full degree sequence $\bd_n$. This rules out the possibility of using specific examples for fragments in our bounds, which was the strategy in the binomial case. We overcome this difficulty, roughly, by classifying fragments into classes depending on the amount of cycles of each length they contain, and obtaining bounds for each whole class, rather than for individual fragments.

In the second part, developed in \Cref{sec:transition}, we establish a sharp threshold phenomenon at $\nu=\nu_0$ for the property that $\Gcal_n(\bd_n)$ has a limit probability set for {\sc FO}-properties that is dense in the unit interval, i.e. $\overline{L(\bd)}=[0,1]$. The constant $\nu_0$ is defined as the unique value of $\nu$ such that a random graph with $\nu=\nu_0$ is acyclic with probability exactly $1/2$ (see \Cref{rem:3}). This is done in \Cref{lem:transition} and its proof is a two-fold application of Kakeya's Criterion, distinguishing the cases $0<\nu < \nu_0$ and  $\nu\geq \nu_0$.

%% file: sections/probability.tex
\section{Cycle distribution}
\label{sec:cycles}

In this section we study the distribution of short cycles in $\Conf(\bd)$. We stress that most of the results presented here are already known and follow from well-established techniques. However these results are scattered through the literature, and this section aims to provide a self-contained compendium of the cycle distribution in the configuration model. We will prove it under \Cref{assump:main}, but it is worth noticing that the assumption (iv) is not needed for these results to hold (and similarly in \Cref{sec:fragment}).

\begin{lemma}
\label{lem:expected_copies_upper}
Let $H$ be a multigraph whose minimum degree is at least $2$. Let $h=|V(H)|$, $h_i=\vert \{ v \in V(H): \deg(v)=i   \} \vert$, and $\ell=|E(H)|$.
Let $X_{n}(H)$ be the number of $H$-copies in $\Conf_n(\bd_n)$. Then
\begin{equation}\label{eq:bound_Xi}
\E{X_{n}(H)} \leq \Xi_n(H) e^{h \xi_n}, 
\end{equation}
where 
\begin{equation}\label{eq:def_Xi}
\Xi_n(H) \coloneqq  
\frac{ (\widehat{n})_h }{\authe(H)\widehat{\lambda}_n^{h} \prod_{i=1}^{\ell} (m_n-2i +1)}\prod_{i\geq 0} \rho_{n,i}^{h_i},
\end{equation}
and $\xi_n=o(1)$ is a sequence depending only on $\Delta_n$ and $n$.
\end{lemma}

{\begin{remark*}
    It is significantly easier to prove~\eqref{eq:bound_Xi} if we replace $\Xi_n(H)$ by
\begin{equation}\label{eq:def_Xi_easier}
\Xi'_n(H) \coloneqq  
\frac{ n^h }{\authe(H) \prod_{i=1}^{\ell} (m_n-2i +1)}\prod_{i\geq 0} \rho_{n,i}^{h_i},
\end{equation}
However, the stronger bound obtained with $\Xi_n(H)$ will be needed to show the tightness of the random variable counting the number of cycles in the configuration model.
\end{remark*}}

\begin{proof}

We estimate the number of possible sub-configurations $H^\prime$ isomorphic to $H$.
Fix a labelling $v_1,\dots, v_h$ of $V(H)$.
In order to choose $H^\prime$, we begin by picking the vertices
$v^\prime_1,\dots, v^\prime_h$ forming $V(H^\prime)$, 
each one labeled after a vertex in 
$H$.

In order to completely determine $H^\prime$, we need to pick
a list of $\deg(v_i)$ half-edges incident to $v_i^\prime$ for each $i \in [h]$. This yields a total of $\prod_{i\in [h]} (a_i)_{\deg(v_i)}$ choices of half-edges for $H^\prime$,
where $a_i= d_{v^\prime_i}$.
Note that this is $0$ unless $d_{v^\prime_i}(n)\geq \deg(v_i)$ for all $i\in [h]$.
There are exactly $\authe(H)$ ways of choosing vertices and half-edges that yield the same sub-configuration $H^\prime$. Hence, the total number of possible sub-configurations of $\Conf_n(\bd_n)$ isomorphic to $H$ is given by
\[
\frac{1}{\authe(H)}
\sum_{a_1,\dots, a_h \in \NN} 
\sum_{\substack{ \{v^\prime_1, \dots, v^\prime_h \}\in \binom{[n]}{h}\\
d_{v^\prime_i} = a_i,\,  i \in [h]
}} \prod_{i \in [h]}  (a_i)_{\deg(v_i)}.
\]

In the sum, we first pick the degrees $a_1,\dots, a_h$ of
$v^\prime_1,\dots, v^\prime_h$ before choosing the vertices themselves.

Given any choice of $a_1,\dots, a_h$,
we define $b_1,\dots, b_k$ as the different numbers 
appearing in $a_1, \dots, a_h$, in increasing order. Observe that $k\leq \Delta_n$.  \par 

Given $i\in [k]$ define $c_i$ as
the number of indices $j\in [h]$ such that $a_j=b_i$. 
Then 
\begin{equation}
\label{eqn:expected_copies_upper_main}
\sum_{\substack{ \{v^\prime_1, \dots, v^\prime_h \}\in \binom{[n]}{h}\\
d_{v^\prime_i} = a_i,  i \in [h]
}} \prod_{i\in [h]}  (a_i)_{\deg(v_i)}= 
\left(\prod_{i\in [k]}  \prod_{0\leq j < c_i} (n_{b_i} - j) \right)
\prod_{i\in[h]}  (a_i)_{\deg(v_i)}.
\end{equation}
We going to apply two technical lemmas, whose statement and proof can be found in \Cref{app:1}. These will allow us to replace the expression inside the parenthesis on the {\sc RHS} of \eqref{eqn:expected_copies_upper_main} by something more convenient. Let $\alpha_i=n_{b_i}$ and $\beta_i=c_i$ for all $i\in [k]$ with $\alpha= \sum_{i\in [k]} \alpha_i$ and $\beta = \sum_{i\in [k]} c_i = h$. By \Cref{lem:aux_ineq},

\begin{align*}
\prod_{i\in [k]} 
\prod_{0\leq j < c_i} 
(n_{b_i} - j) & \leq 
(\alpha)_h 
\left(
\prod_{i\in [h]} 
\frac{n_{a_i}}{\alpha}
\right)
\left( 
\prod_{h - k + 1 \leq j < h}\frac{\alpha}{\alpha-j}
\right) \\
&  \leq 
(\widehat{n})_h 
\left(
\prod_{i\in [h]} 
\frac{n_{a_i}}{\widehat{n}}
\right)
\left( 
\prod_{h + 1 - t \leq j < h}\frac{\widehat{n}}{\widehat{n}-j}
\right),
\end{align*}
where $t=\min (\Delta_n,h+1)$. For the second inequality we used that $\alpha\leq \widehat{n}$ (since $H$ has minimum degree $2$) and that $k\leq t$. Applying \Cref{lem:aux_ineq2} with $N=\widehat{n}$ and $a=h$, we obtain
\begin{align*}
 \prod_{i\in [k]} 
\prod_{0\leq j < c_i} 
(n_{b_i} - j)   
& \leq 
(\widehat{n})_h 
\left(
\prod_{i\in [h]} 
\frac{n_{a_i}}{\widehat{n}}
\right)
e^{h\xi_n} =
(\widehat{n})_h \widehat{\lambda}_n^{-h}
\left(
\prod_{i\in [h]} 
\frac{n_{a_i}}{n}
\right)
e^{h\xi_n}, 
\end{align*}
for some sequence $\xi_n=o(1)$ that only depends  on $n$ and $\Delta_n$.

Replacing this last inequality into~\eqref{eqn:expected_copies_upper_main} we get
\[
\sum_{\substack{ \{v^\prime_1, \dots, v^\prime_h \}\in \binom{[n]}{h}\\
d_{v^\prime_i} = a_i,  i \in [h]
}} \prod_{i\in [h]}  (a_i)_{\deg(v_i)} \leq 
(\widehat{n})_h \widehat{\lambda}_n^{-h}
\left(
\prod_{i\in [h]} \frac{n_{a_i}}{n}  (a_i)_{\deg(v_i)}\right)  e^{\xi_n h}.
\]
Summing over all possible choices of $a_1,\dots, a_h$, the number of possible sub-configurations of $\Conf_n(\bd_n)$ isomorphic to $H$ is bounded from above by
\[
\frac{(\widehat{n})_h }{\authe(H)\widehat{\lambda}_n^h}
\left(
\prod_{i\in [h]} \mmnt_{n,\deg(v_i)} \right) e^{\xi_n h}.
\]
Finally, the probability that a given copy of $H$ is realized is $1/\prod_{i\in [\ell]} (m_n-2i +1)$.
We conclude that
\begin{align*}
\E{X_{n}(H)} & \leq \frac{(\widehat{n})_h}{\authe(H)\widehat{\lambda}_n^h \prod_{i\in [\ell]} (m_n-2i +1)}
\left(
\prod_{i\in [h]} \mmnt_{n,\deg(v_i)} \right) e^{\xi_n h} &  \\
& = \Xi_n(H) e^{\xi_n h}. &
\end{align*}
\end{proof}

\begin{lemma}
    \label{lem:expected_copies}
    Let $H$ be a multigraph. Using the notation of \Cref{lem:expected_copies_upper}, it holds that 
    \[
    \E{X_{n}(H)} = \left(1 + O(1/n)\right)\frac{n^{h}}{m_n^{\ell}}\prod_{i\geq 0} \mmnt_{n,i}^{h_i}.
    \]
\end{lemma}
\begin{proof}
Recall that we can write
\begin{equation*}
\E{X_{n}(H)}=\frac{1}{\authe(H)\prod_{i\in [\ell]} (m_n-2i +1)}
\sum_{a_1,\dots a_h \in \NN} 
\sum_{\substack{ \{v^\prime_1, \dots, v^\prime_h \}\in \binom{[n]}{h}\\
d_{v^\prime_i} = a_i,  i \in [h]
}} \prod_{i\in[h]}  (a_i)_{\deg(v_i)}.
\end{equation*}
It holds that
\begin{align*}
\prod_{i\in[h]}  (n_{a_i} - h) (a_i)_{\deg(v_i)}
 & \leq 
 \!\!\!
\sum_{\substack{ \{v^\prime_1, \dots, v^\prime_h \}\in \binom{[n]}{h}\\
d_{v^\prime_i} = a_i,  i \in [h]
}} \prod_{i\in[h]}   (a_i)_{\deg(v_i)}
 \leq 
\prod_{i\in[h]}  n_{a_i} (a_i)_{\deg(v_i)}.
\end{align*}
Hence, we obtain the desired result
\begin{align*}
\E{X_{n}(H)} & =
\left(1 + O\left(1/n\right)\right)
\frac{1}{\authe(H) m_n^\ell} \sum_{a_1,\dots ,a_h \in \NN} 
\prod_{i\in[h]}  n_{a_i} (a_i)_{\deg(v_i)}\\
& =
\left(1 + O\left(1/n\right)\right)
\frac{n^h}{\authe(H) m_n^\ell} 
\prod_{i\in[h]}  \sum_{a_i\in \NN}  \frac{n_{a_i}}{n} (a_i)_{\deg(v_i)}\\
& = \left(1 + O\left(1/n\right)\right)
\frac{n^h}{\authe(H) m_n^\ell} 
\prod_{i\in[h]}  \mmnt_{n,\deg(v_i)}^{h_i}.\qedhere
\end{align*}
\end{proof}

From the previous lemma, we recover a classic result on random graphs.
\begin{lemma}
\label{thm:cycle_distribution}
Let $X_{n,k}$ be number of $k$-cycles in $\Conf_n(\bd_n)$.
Then, for any finite collection 
$k_1, \dots, k_l$, the
variables $X_{n,k_1},\dots,
X_{n,k_l}$ converge in distribution to independent Poisson variables whose respective means are $\xi_{k_i}:=
\nu^{k_i}/2k_i$. In particular, 
\begin{equation}
\label{eq:prob_being_simple}
\lim_{n\to \infty} \Pr(\Conf_n(\bd_n)\text{ is simple})= e^{-\frac{\nu}{2} - \frac{\nu^2}{4}}.
\end{equation}
\end{lemma}
\begin{proof}
We prove the first part of the statement, the second part follows easily from it.  
We assume $\mmnt_2>0$. Otherwise, by \Cref{assump:main}.(iv), all
vertices have degree $0$ or $1$ and the result follows trivially. 
By the method of moments (\Cref{thm:moments_method}), it suffices to show that for any 
$a_1,\dots, a_l\in \NN$,
\begin{equation}
\label{eq:mom}
\lim_{n\to \infty}
\E{
\prod_{i=1}^l \binom{X_{n,k_i}}{a_i}
} =
\prod_{i=1}^l \frac{\xi_{k_i}^{a_i}}{a_i!}.
\end{equation}
We say that a multigraph $G$ is a \defin{non-degenerate} union of unlabeled multigraphs $H_1,\dots, H_t$ if
$G$ contains a copy $H^\prime_i$ of $H_i$ for each $i\in [t]$,
$V(G)=V(H^\prime_1)\cup \dots \cup V(H^\prime_t)$ (note that this union is not necessarily disjoint),
and the $H_i$ are pairwise different  (note that the $H_i'$ can share some edges). Let $\Hcal$
be the class of all unlabeled multigraphs
that are non-degenerate unions 
of $a_1$ copies of $C_{k_1}$,
$a_2$ copies of $C_{k_2}$, and so on. Let $H_*\in \Hcal$ be the graph formed by the disjoint union of the corresponding cycles.
Note that $\ex(H_*)=0$ and $\ex(H)>0$ for all other $H\in \Hcal$.
Given $H\in \Hcal$, let $Y_{n}(H)$ be the 
number of $H$-copies in $\Conf_n(\bd_n)$. The {\sc LHS} in \eqref{eq:mom} amounts to
$\sum_{H\in \Hcal} \frac{\authe(H)}{\authe(H_*)}\E{Y_{n}(H)}$. 
We show that asymptotically only
$\E{Y_n(H_*)}$ contributes to the value of this
sum, and this expectation has the desired value.
Using 
\Cref{lem:expected_copies} we get
\begin{equation}
\label{eq:expect}
\E{Y_{n}(H)}
= (1 + O(1/n)) \frac{n^{h} \, \prod_{i\in \NN}
(\mmnt_{n,i})^{h_i}
}{
(m_n)^{\ell}\, \authe(H)}=
O\left(
n^{-\ex(H)}
\prod_{i\geq 0} \mmnt_{n,i}^{h_i}
\right),
\end{equation}
where $h=|V(H)|, \ell=|E(H)|$, and
$h_i= | \{ 
v\in V(H):\deg(v)=i \} |$. \par
Let
$H\in \Hcal$ be an arbitrary multigraph different from $H_*$. 
We show that $\E{Y_{n}(H)}=o(1)$.
By \Cref{assump:main},
$\Delta_n=o(n^{1/2})$ and $\mmnt_{n,2}=O(1)$. The former implies that $\mmnt_{n,i}=o(n^{1/2} \mmnt_{n,i-1})$
for all $i\geq 2$, and then $\mmnt_{n,i}=o(n^{(i-2)/2})$ for all $i\geq 3$.
As $H\neq H_{*}$, it contains some vertex of degree at least $3$ and we obtain
\begin{equation}\label{eq:excess}
\prod_{i\geq 0} \mmnt_{n,i}^{h_i}= o\left( 
\prod_{i\geq 3} n^{h_i(i-2)/2}
\right) = o\left( n^{\ex(H)}
\right) ,
\end{equation}
where we used that $\sum_{i\geq 3} h_i(i-2) = \sum_{i\geq 0} h_i(i-2) = 2\ex(H)$, as $H$ has minimum degree at least $2$. Plugging this last equation into \eqref{eq:expect}, we obtain that $ \E{Y_{n}(H)}=o(1)$.

Now consider the case $H=H_*$. Since $H_*$ is $2$-regular and $\authe(H_*)=
\prod_{i=1}^l a_i!(2k_i)^{a_i}$, \Cref{lem:expected_copies}
yields
\begin{equation}\label{eq:H*}
\E{Y_{n}(H_*)}= \prod_{i=1}^l
\frac{\xi_{k_i}^{a_i}}{a_i!}  + o(1),
\end{equation}
using that $m_n/n=\mmnt_{n,1}$. Combining \eqref{eq:expect}--\eqref{eq:H*}, we obtain the first part of the lemma
\[
\lim_{n\to \infty}
\E{
\prod_{i=1}^l \binom{X_{n,k_i}}{a_i}
} = 
\lim_{n\to \infty} \E{Y_{n}(H_*)}=
\prod_{i=1}^l \frac{\xi_{k_i}^{a_i}}{a_i!}.
\]
\end{proof}

Once we have determined the distribution of short cycles, we proceed to study the probability of acyclicity. Recall that $\nu=1$ is the threshold for the existence of a giant component. If $\nu\geq 1$, the largest component w.h.p. contains an unbounded number of cycles and $p_\acyc(\bd)=0$. Thus we restrict to the subcritical case $\nu<1$.

\begin{lemma}
\label{lem:tight_cycles}
    Fix $k\in \NN$. Let $X_{n,k}$ count the number of $k$-cycles in $\Conf_n(\bd_n)$. Assume that $\nu<1$. Then the
    sequence $(\E{X_{n,k}})_{n\in \NN}$ is tight.
\end{lemma}

\begin{proof}
    Clearly, adding or removing isolated
vertices to $\Conf_n(\bd_n)$ does not affect the result, so without loss of generality we may assume $\lambda_0=0$. If $\lambda_1=1$, by \Cref{assump:main}.(iv) all vertices have degree $1$ and the result follows trivially as there are no cycles in the model. Also, note that $\lambda_1\neq 0$, as this together with $\lambda_0=0$ would imply that $\nu\geq 1$.  So we assume that $0<\lambda_1<1$.
From \Cref{lem:expected_copies_upper} it 
follows that
\begin{align}\label{eq:SDJSO}
\E{X_{n,k}} \leq \; &
\frac{(\widehat{n})_k }{2k \widehat{\lambda}_n^k \prod_{i=1}^{k} 
(m_n - 2i +1)} \mmnt_{n,2}^{k} e^{k \xi_n }
\end{align}
for all $n,k\in \NN$, where the sequence $\xi_n =o(1)$ depends only on $n$. Observe that $m_n\geq 2\widehat{n} + \lambda_{n,1} n$, so 
$m_n - 1 \geq 2\widehat{n}$ for sufficiently large $n$. This implies that 
$(\widehat{n}-s)/(m_n-2s-1)\leq \widehat{n}/m_n= \widehat{\lambda}_n/\mmnt_{n,1}$ for all $0\leq s < n$. Thus, for sufficiently large $n$, 
\[
\E{X_{n,k}} \leq \frac{(\nu_n e^{\xi_n })^{k}}{2k}.
\]
Choose $\nu^\prime\in (\nu,1)$.
As $\Ln \nu_n e^{\xi_n }=\nu$, 
there is some value $n^\prime\in \NN$ such that $\nu_n e^{\xi_n }<\nu^\prime$
for all $n\geq n^\prime$. Then
\[
\E{X_{n,k}} \leq 
\frac{(\nu^\prime)^k }{2k}
\]
for all $n\geq n^\prime$. 
As the sum $\sum_{k\geq 1} \frac{(\nu^\prime)^k}{2k}$ converges,
this proves the result.
\end{proof}

\begin{corollary}
\label{lem:cycle_count_conf}
Let $Z_n$ count the cycles in $\Conf_n(\bd_n)$. Assume that $\nu<1$. Then,
\begin{equation}
\lim_{n\to\infty} \E{Z_n}= -\frac{1}{2} \, \ln{(1-\nu)}.
\end{equation}
\end{corollary}
\begin{proof}
Let $X_{n,k}$ count the number of $k$-cycles in
$\Conf_n(\bd_n)$. In \Cref{thm:cycle_distribution} we showed that $\E{X_{n,k}}=
\nu^k/2k + o(1)$. By \Cref{{lem:tight}} and \Cref{lem:tight_cycles},
\begin{equation}
\label{lem:cycle_count_conf.eqn:main}
\lim_{n\to\infty} \E{Z_n} = \lim_{n\to\infty} \sum_{k\geq 1}\E{X_{n,k}} = \sum_{k\geq 1} \lim_{n\to\infty} \E{X_{n,k}} = -\frac{1}{2} \, \ln{(1-\nu)}.
\end{equation} \par

\end{proof}

\begin{lemma}
\label{cor:acyclic}
Assume $\nu<1$. Let $\bm{a}=(a_\ell)_{\ell\in \NN}$ be a sequence of non-negative integers such that $\sum_{\ell \in \NN} a_\ell< \infty$. 
The following hold true:
\begin{itemize}

    \item[(1)]
    Let $A_n$ be the event that $\Conf_n(\bd_n)$
    contains exactly $a_\ell$ $\ell$-cycles for all $\ell\geq 1$. Then 
    \[
    \lim_{n\to \infty}\Pr(A_n)=
    \sqrt{1-\nu} \prod_{\ell\in \NN}
    \frac{(\nu^\ell/2\ell)^{a_\ell}}{a_\ell!}.
    \]
    In particular, 
    $$
    \Ln \Pr(\Conf_n(\bd_n) \text{ is acyclic}) = \sqrt{1-\nu}.
    $$
    \item[(2)]
    Let $B_n$ be the event that $\Gcal_n(\bd_n)$
    contains exactly $a_\ell$ $\ell$-cycles for all $\ell\geq 3$. Then 
    \[
    \lim_{n\to \infty}\Pr(B_n)=
    \sqrt{1-\nu} \cdot e^{\frac{\nu}{2} + \frac{\nu^2}{4}} \prod_{\ell\geq 3}
    \frac{(\nu^\ell/2\ell)^{a_\ell}}{a_\ell!}.
    \]
    In particular, 
    $$
    \Ln \Pr(\Gcal_n(\bd_n) \text{ is acyclic}) = \sqrt{1-\nu}\cdot e^{\frac{\nu}{2} + \frac{\nu^2}{4}}.
    $$
\end{itemize}
\end{lemma}
\begin{proof}
We will prove (1). Statement (2) follows from the fact that $\Gcal_n(\bd_n)$ is distributed like $\Conf_n(\bd_n)$ 
conditioned on the absence of $1$-cycles and $2$-cycles, whose probability we computed in \Cref{{thm:cycle_distribution}}. 
Let $X_{n,k}$ count the number of $k$-cycles in $\Conf_n(\bd_n)$ and $\xi_k=\nu^k/2k$. 
Let $\epsilon>0$ be arbitrarily small. 
It suffices to prove that, if $n$ is large enough,
\begin{equation}
\label{eqn:acyclic_aux}
 \left\vert
\Pr(A_n) -  \sqrt{1-\nu} \prod_{\ell\in \NN}
    \frac{\xi_\ell^{a_\ell}}{a_\ell!}
\right\vert < \epsilon. 
\end{equation}
Let $K$ be a sufficiently large constant satisfying both
\begin{eqnarray}
&\left\vert
\left(\prod_{k = 1}^K
    e^{-\xi_k}\cdot \frac{\xi_k^{a_k}}{a_k!} \right)
- \left(
\sqrt{1-\nu}
\prod_{k \geq 1} \frac{\xi_k^{a_k}}{a_k!}
\right)
\right\vert < \epsilon/3 \label{eq:cond_1}
\end{eqnarray}
and
\begin{eqnarray}
&\Pr\big(\sum_{k > K} X_{n,k}>0\big)<\epsilon/3 ,\quad\text{ for all $n\in \NN$}. \label{eq:cond_2}
\end{eqnarray}

The property in \eqref{eq:cond_1} can be attained for a large $K$ because, inside the absolute value, 
for $\nu<1$, the right term contains the limit 
(as $K\to \infty$) of the left term. 
The existence of $K$ satisfying \eqref{eq:cond_2}
follows from 
$(\E{X_{n,k}})_{n \in \NN}$ being tight for all $k\in \NN$,
as shown in \Cref{lem:cycle_count_conf}, and Markov's inequality.
Indeed, by tightness, there is some $K$ for which
$\sum_{k>K}\E{X_{n,k}}< \epsilon/3$ uniformly in $n$, and then 
$\Pr\big(\sum_{k > K} X_{n,k}> 0\big)\leq \E{\sum_{k>K}X_{n,k}}\leq \epsilon/3$ uniformly in $n$ as well. 
\par

We can write $A_n= \bigcap_{k\geq 1} \{X_{n,k}=a_k\}$, and $\Pr(A_n)= \lim_{k\to \infty} p_{n,k}$,
where $p_{n,k}= \Pr(
\bigcap_{i=1}^k
\{X_{n,i}=a_i\})$.
The intersection bound and \eqref{eq:cond_2} imply
that $\Pr(A_n)> p_{n,K} -\epsilon/3$ for all $n\in \NN$.
Moreover, the sequence $(p_{n,k})_{k\geq 1}$ is monotonically decreasing, so $\Pr(A_n)\leq p_{n,K}$ for all $n\in \NN$. 
It follows that,
\begin{equation}
\lvert \Pr(A_n) - p_{n,K} \rvert < \epsilon/3
\end{equation}

However, by \Cref{thm:cycle_distribution}, for $n$ large enough
\[
\left|p_{n,K} - 
\prod_{k = 1}^K
   e^{-\xi_k}
    \frac{\xi_k^{a_k}}{a_k!}\right|
<  \epsilon/3.\]
Using \eqref{eq:cond_1} here yields \eqref{eqn:acyclic_aux} and completes the proof.
\end{proof}

We finally show that there are no complex components (i.e., components with positive excess). 

\begin{theorem}
\label{thm:complex}
Assume $\nu<1$. Then a.a.s there are no connected components containing more than one cycle in $\Conf_n(\bd_n)$.
\end{theorem}
\begin{proof}
As in the proof of \Cref{lem:tight_cycles}, we will assume that $\lambda_0=0$ and $\lambda_1<1$, which in turn imply that $\mmnt_1>1$ and $n-i/m_n - 2i + 1 \leq 1/\mmnt_{n,1}$ for sufficiently large $n\in \NN$ for all $i\geq 1$.
\par
The configuration $\Conf_n(\bd_n)$ has two cycles lying in the same component if and only if it has some subgraph belonging to one of the following classes:
\begin{enumerate}[label=(\Roman*)]
    \item $H^{(1)}_{i,j,k}:$ An $i$-cycle
    and a $j$-cycle, disjoint, with a path of length $k\geq 1$ joining a vertex from each cycle. 
    \item $H^{(2)}_{i,j,k}:$ An $i$-cycle
    and a $j$-cycle sharing a path of length $k\geq 1$.
    \item $H^{(3)}_{i,j}:$
    An $i$-cycle and a $j$-cycle sharing a single vertex.
\end{enumerate}

We show that w.h.p. none of these subgraphs appear in $\Conf_n(\bd_n)$ using the first moment method. 
Let us consider Class-(I) first. 
Let $X^{(1)}_{n;i,j,k}$
count the number of copies of 
$H^{(1)}_{i,j,k}$
in $\Conf_n(\bd_n)$.
The multigraph $H^{(1)}_{i,j,k}$
has $i+j+k$ edges and $h=i+j+k-1$ vertices, among which two have degree $3$ and the rest have degree $2$.
Observe that if $h> \widehat{n}$ then $\E{X^{(1)}_{n;i,j,k}}= 0$. Suppose otherwise.
Choose some $\nu^\prime \in (\nu, 1)$. By \Cref{lem:expected_copies}, for sufficiently large $n$, independently of $h$
\begin{align}
\nonumber
\E{X^{(1)}_{n;i,j,k}} & \leq \; 
\frac{(\widehat{n})_{h}\widehat{\lambda}_n}{\prod_{s\in [h+1]} 
(m_n - 2s +1)} \mmnt_{n,2}^{h-2}
\mmnt_{n,3}^{2} e^{h\xi_n }  \\
\nonumber
& \leq (\nu_n)^{h}\frac{\Delta_n^2}{(m_n - 2h -1)} e^{h\xi_n } \\
 &  \leq  (\nu^\prime_n)^{h}\frac{\Delta_n^2}{(\lambda_{n,1}n -1)},
\end{align}
where in the second inequality we have used that $\mmnt_{n,3} \leq \Delta_n \mmnt_{n,2}$.
\par

Let $X^{(1)}_{n;h}$ be the sum of all variables $X^{(1)}_{n;i,j,k}$ with
$i+j+k-1=h$. There are at most $h^2$ such choices for $i,j,k$, 
so
$\E{X^{(1)}_{n;h}}= O(h^2 (\nu_n^\prime)^h \Delta_n^2/n)
$. It follows that
\[
\sum_{h\geq 3} \E{X^{(1)}_{n;h}}
= O\Big(
\frac{\Delta_n^2}{n}  
\sum_{h\geq 3}h^2 (\nu_n^\prime)^h 
\Big)=O\Big(
\frac{\Delta_n^2}{n} 
\Big) = o(1).
\]
Using Markov's inequality we obtain that w.h.p. no
Class-(I) subgraph occurs in $\Conf_n(\bd_n)$. In an analogous way, it can be shown that the expected numbers of Class-(II) and Class-(III) subgraphs in $\Conf_n(\bd_n)$ 
are $O\left(\Delta_n^2/n\right)$, and w.h.p. $\Conf_n(\bd_n)$ contains no subgraph from those subclasses either. This proves the result. 
\end{proof}

\subsection{Previous results}\label{sec:previous}

In this section we included a self-contained description of the cycle distribution on $\Conf(\bd)$ and the probability $\Conf(\bd)$ and $\Gcal(\bd)$ are acyclic. As we explained, most of these results were already known in the literature and our aim was to compile them in a single document. To conclude this section, we give references where some of these results can be found. Note that all the results that we refer to below, hold w.h.p. and under similar conditions on $\bd$ as \Cref{assump:main}, unless otherwise stated. 

The expected number of copies of a given small subgraph (\Cref{lem:expected_copies}), as well as the probability a particular cycle appears, are well-studied in the literature~\cite{bollobas1986number}. The result stated here can be found with a similar proof in the unpublished notes by Bordenave~\cite[Theorem 2.4]{Bordenave}. The refined upper bound (\Cref{lem:expected_copies_upper}), needed to show tightness (\Cref{lem:tight_cycles}), was not available in the literature, to our best knowledge.

For graphs other than trees or unicyclic ones, their expected number tends to zero as $n$ tends to infinity.   McKay~\cite{mckay2011subgraphs} studied the probability of small subgraph appearance on dense random graphs with given degrees.

That the joint distribution of cycles up to length $k$  converges to a vector of independent Poisson random variables is well known. This was originally proved  in regular setting by~Bollob\'as~\cite{bollobas1980probabilistic}, and, independently, by Wormald~\cite{wormald1981asymptotic}. The extension to the configuration model can be found in~\cite[Theorem 2.18]{Bordenave} and~\cite[Exercise 4.16]{vanderhofstadRandomGraphsComplexVolI}. 
Determining the probability that the configuration model is simple is one of earliest results in the area~\cite{benderAsymptoticNumberLabeled1978,bollobas1980probabilistic}, see \cite{janson2009probability} for the version in \Cref{thm:simple}.

In the subcritical regime, the result that all cycles have length bounded in probability (\Cref{{lem:cycle_count_conf}}), which is equivalent to the tightness exhibited in \Cref{lem:tight_cycles}, can be found in~\cite[Lemma 5.3]{dhara2017phase}. The existence of no complex components in $\Conf_n(\bd_n)$, given by \Cref{thm:complex}, can be found in~\cite{hatami2012scaling}. It is worth stressing that the proof strategy there is quite different: they analyze a process that exposes a connected component in $\Conf_n(\bd_n)$, vertex by vertex, using martingale concentration inequalities, which requires additional constraints in the maximum degree, with respect to \Cref{assump:main}.

%% file: sections/fragment.tex
\section{Fragment Distribution}
\label{sec:fragment}

In this section we study the distribution of unicyclic components of $\Conf_n(\bd_n)$ when $\nu<1$. To our best knowledge, this distribution has not been studied yet.

Let $\Frag^*_n=\Frag(\Conf_n(\bd_n))$ 
and $\Frag_n=\Frag(\Gcal_n(\bd_n))$ be the subgraphs composed by the union of all unicyclic connected components in $\Conf_n(\bd_n)$ and $\Gcal_n(\bd_n)$, respectively.

\begin{theorem}
\label{thm:frag_distribution}
Suppose $\nu<1$. If $H$ is a fragment, then
    \[
    \lim_{n\to \infty}
    \Pr(\Frag^*_n \simeq H) =
    \frac{\sqrt{1-\nu}}{\authe(H)}
    \prod_{i\geq 0} 
    \left(\frac{\lambda_i i!}{\mmnt_1}
    \right)^{h_i},
    \]
    where $h_i=| \{ 
    v\in V(H) : \deg(v)=i\}|$.

\end{theorem}
\begin{proof}

Let  $h=|V(H)|$ and 

$V(H)=\{ v_1,\dots, v_h\}$. For each $v_i\in V(H)$, fix some ordering of
the half-edges incident to $v_i$.
Define $\Hcal_n$ as the set of possible isolated $H$-copies
in $\Conf_n(\bd_n)$. In order to pick a copy $H^\prime \in \Hcal_n$, 
we first select the vertices $v^\prime_1,\dots, v^\prime_h$. As we want the copy to be isolated, we require $d_{v^\prime_i}= \deg(v_i)$ for all $i\in [h]$. In order to completely determine $H^\prime$ we give an ordering of the half-edges incident to each vertex $v^\prime_i$. Afterwards, half-edges should be matched according to the half-edge orderings defined for $H$. Observe that there are exactly $\authe(H)$ ways of picking vertices and half-edge orderings that
result in the same subconfiguration $H^\prime$. Hence,
\begin{equation}
\label{eq:size_H_n}
|\Hcal_n|= \frac{\prod_{i\geq 0} (n_i)_{h_i} (i!)^{h_i}}{\authe(H)}.
\end{equation}
Given $H^\prime \in \Hcal_n$, let $A_n(H^\prime)$ be the event 

that $H^\prime \subseteq \Conf_n(\bd_n)$ and $\Conf_n(\bd_n)\setminus V(H^\prime)$ is acyclic. Observe that the events $A_n(H^\prime)$ are disjoint. 
Let $P_n$ be the event that no component in $\Conf_n(\bd_n)$ contains more than one cycle. Then the event $(\Frag^*_n\simeq H) \cap  P_n$ coincides with the union of the events $A_n(H^\prime)$ for all $H'\in \Hcal_n$. Thus,
by \Cref{thm:complex},
\begin{equation}
\label{eq:frag_prob1}
 \Pr(\Frag^*_n \simeq H)= o(1)+ \sum_{H^\prime \in \Hcal_n} \Pr(A_n(H^\prime) ).
\end{equation}
Recall that $d_{v_i'}=\deg(v_i)$ for all $H'\in \Hcal_n$ and all $i\in [h]$. Thus, by symmetry, the probability of $A_n(H^\prime)$ is the same for all $H^\prime \in \Hcal_n$. Fix an $H$-copy $H^\prime \in \Hcal_n$ for each $n\in \NN$. Combining \eqref{eq:size_H_n} and \eqref{eq:frag_prob1} we obtain
\begin{equation}
\label{eq:frag_prob2}
\Pr(\Frag^*_n \simeq H) = \frac{\prod_{i\geq 0} (n_i)_{h_i} (i!)^{h_i}}{\authe(H)} \Pr(A_n(H^\prime))+ o(1).
\end{equation}
Let us examine now the probability of $A_n(H^\prime)$.
Let $\widehat{G}_n= \Conf_n(\bd_n)[[n]\setminus V(H^\prime_n)]$.
As $H$ is a fragment, by definition
\begin{equation}
\label{eq:frag_prob3}
\Pr(A_n(H^\prime))= \frac{1}{\prod_{i=1}^{h}
(m_n - 2i + 1)} \Pr\left( \widehat{G}_n \text{ is acyclic} \mmid 
H^\prime \subseteq \Conf_n(\bd_n) \right).
\end{equation}
Let $\widehat{\bd}_{n-h}$ be the degree sequence obtained by removing the vertices of $V(H^\prime_n)$ from $[n]$ and relabeling the remaining vertices as
$[n-h]$. Note that $(\widehat{G}_n \mid H^\prime_n \subseteq \Conf_n(\bd_n)) \sim \Conf_{n-h}(\widehat{\bd}_{n-h})$. Clearly, the degree sequence
$\widehat{\bd}$ satisfies \Cref{assump:main}. Moreover, 
it is easy to see that the first and second moments of the related degree distribution have the same limits as those of $\bd$ (that is, $\mmnt_1$ and $\mmnt_2$) as $h=O(1)$. 
By \Cref{cor:acyclic},
\begin{equation}
\label{eq:frag_prob4}
\begin{aligned}
\Pr\left( \widehat{G}_n \text{ is acyclic} 
 \mmid H^\prime_n \subseteq \Conf_n(\bd_n) \right)
 &=\Pr\left(\Conf_{n-h}(\widehat{\bd}_{n-h})\text{ is acyclic}  \right)\\
 &= \sqrt{1-\nu} + o(1).
\end{aligned}
\end{equation}
Putting \eqref{eq:frag_prob2}, \eqref{eq:frag_prob3} and \eqref{eq:frag_prob4} together,
we conclude the proof of the theorem
\begin{equation}
\label{eq:frag_final}
\begin{aligned}
\Pr(\Frag^*_n \simeq H)  
& =  
\frac{ \sqrt{1-\nu} \prod_{i\geq 0} (n_i)_{h_i} (i!)^{h_i}}{\authe(H) \prod_{i=1}^h
(m_n - 2i + 1)} + o(1)   \\ 
& =   
\frac{\sqrt{1-\nu}}{\authe(H)}\prod_{i\geq 0}
\left(\frac{\lambda_i i!}{\mmnt_1}
\right)^{h_i} + o(1).\qedhere
\end{aligned}
\end{equation}
\end{proof}

The following corollary states that the fragment of $\Gcal_n(\bd_n)$
is asymptotically distributed like the fragment of
$\Conf_n(\bd_n)$, ignoring the components containing loops or double edges.

\begin{corollary}\label{cor:prob_frag}
Assume that $\nu<1$. Let $G$ be a simple 
fragment. 
Then 
\[
\lim_{n\to \infty}
\Pr(\Frag_n \simeq G) =
\frac{\sqrt{1-\nu}\cdot e^{\frac{\nu}{2} + \frac{\nu^2}{4} }}{\aut(G)}
\prod_{i\geq 0} 
\left(\frac{\lambda_i i!}{\mmnt_1}
\right)^{g_i},
\]
where $g_i=| \{v\in V(G) : \deg(v)=i\}|$.
\end{corollary}
\begin{proof}
Let $A_n$ be the event that $\Conf_n(\bd_n)$ is simple (i.e., it contains neither loops nor multiple edges).
By definition, 
\[
\Pr(\Frag_n \simeq  G)= \Pr(\Frag^*_n \simeq G \, \mid \, 
A_n) = \frac{\Pr(\Frag^*_n \simeq G \, \cap \, 
A_n)}{\Pr(A_n)}.
\]
When $\Conf_n(\bd_n)$ has no complex components,
the event $(\Frag^*_n \simeq G) \, \cap \, A_n$ is equivalent
to $\Frag^*_n \simeq G$. Therefore, using \Cref{thm:complex}, we obtain
\[
\Pr(\Frag_n \simeq G)= \frac{\Pr(\Frag^*_n \simeq G)}{
\Pr(A_n)} + o(1).
\]
By \Cref{lem:cycle_count_conf}, $\Pr(A_n)=
e^{-\nu/2 - \nu^2/4}+ o(1)$. This, together with the previous
theorem and the fact that $\aut(G)=\authe(G)$ when $G$ is simple, proves the result. 
\end{proof}

From now on let $p^*_{n}(H)= \Pr(\Frag^*_n \simeq H)$, 
$p_n(G)= \Pr(\Frag_n \simeq G)$, \linebreak
$p^*(H)= \lim_{n\to \infty} p^*_n(H)$,
and $p(G)=\lim_{n\to \infty} p_n(G)$, for all unlabeled fragments $H$, and all unlabeled simple fragments $G$. We recall that, as proven in \Cref{thm:frag_distribution},
\[
p^*(H) = 
    \frac{\sqrt{1-\nu}}{\authe(H)}
    \prod_{i\geq 0} 
    \left(\frac{\lambda_i i!}{\mmnt_1}
    \right)^{h_i},
    \]
    where $h_i=| \{ 
    v\in V(H) : \deg(v)=i\}|$.
Our next goal is to show that the numbers $p^*(H)$ define a distribution over unlabeled fragments, that is
$\sum_{H} p^*_H=1$. 

\begin{definition}
    A \defin{lexicographically labeled forest (LLFo)} is a rooted forest $F$ such that
\begin{itemize}
    \item[(1)] if $F$ has $r$ components, the roots $v_1,\dots v_r$ are labeled by $[r]$
    in an arbitrary way;
    \item[(2)] if a vertex $v$ with label $\ell\in \NN^*$ has children $v_1, \dots, v_j$ (in some arbitrary order), then they are labeled by
    $\ell 1, \dots, \ell j$ respectively, where by $\ell i$ we mean the concatenation of the label $\ell$ with $i$. 
\end{itemize}    
Given an integer $K\geq 1$, we use $\TT^K_\text{lex}$ to denote the set of LLFo with at least $K$ components.
\end{definition}

Let $\bm{D}=(D_r,D)$ where $D_r$ and $D$ are two probability distributions
over non-negative integers,
and let $F$ be an LLFo. We define 
\begin{equation}\label{eq:def_prob_trees}
p^{\bm{D}}(F): = \prod_{i\geq 0}
\Pr(D_r=i)^{f^r_i} \prod_{i \geq 1} 
\Pr(D = i-1)^{f_i},
\end{equation}
where $f^r_i$ denotes the number of roots in $F$ whose degree is $i$ and $f_i$ denotes the number of non-root vertices in $F$ of degree $i$.  Equivalently, $p^{\bm{D}}(F)$ is the probability that $F$ is generated by
a branching process where
the first generation has $r$ individuals whose offspring is distributed as $D_r$, and the other elements have offspring distribution $D$. It is well known that such process has extinction probability $1$ whenever $\E{D}<1$ \cite[Theorem 3.1]{vanderhofstadRandomGraphsComplexVolII}. We can rephrase that fact as follows.

\begin{lemma}
\label{lem:branching_process}
Let $\bm{D}=(D_r,D)$ where $D_{r}$ and $D$ are two probability distributions over non-negative integers satisfying $\E{D}<1$. Fix an integer $K \geq 1$. 
Then
\[
\sum_{T\in \TT^K_\text{lex}} p^{\bm{D}}(F) = 1.
\]

\end{lemma}

Fragments can be seen as collections of cycles where a tree ``grows out'' from each vertex. Therefore, we can imagine the edges of the trees as being oriented towards the cycle in their connected component. For a vertex $v$ in the fragment that is not in any of the cycles, we define its \defin{parent} to be its unique vertex $u$ that $v$ points to.

We now extend the definition of LLFo to fragments.
\begin{definition}
    A \defin{lexicographically labeled fragment (LLFr)} $H$
    is a fragment where 
\begin{itemize}
    \item[(1)] if $H$ has $r$ cycles, they are labeled by $[r]$ in a non-decreasing order according to their lengths (ties are resolved arbitrarily); 
    \item[(2)] if $v_1,\dots,v_k$ are the vertices in a $k$-cycle with label $\ell\in [r]$, they are labelled by $\ell 1, \dots, \ell k$ following an arbitrary cyclic ordering; 
    \item[(3)] if a vertex $v$ with label $\ell\in \NN^*$ has children $v_1, \dots, v_j$, they are labeled by $\ell 1, \dots, \ell j$ following some arbitrary order.
\end{itemize}
\end{definition}

Next lemma computes the number of LLFr isomorphic to a given fragment.
\begin{lemma}
\label{lemma:numb_iso_LLFr}
    Let $H$ be a fragment. For each $k\geq 1$, let $a_k$
    be the number of $k$-cycles in $H$. For each $i\geq 1$, let $h^r_i$ be the number of vertices of degree $i$ lying in some cycle of $H$, and let $h_i$ be the number of vertices of degree $i$ that are not in any cycle of $H$. Then the number of LLFr isomorphic to $H$ is
    \begin{equation}
    \label{eq:number_of_labelings}
    \gamma(H):= \frac{1}{\authe(H)} \left(\prod_{k\geq 1} a_k! (2k)^{a_k} \right)\left(
\prod_{i\geq 2} (i-2)!^{h^r_i} \right) \left( \prod_{i\geq 1} (i-1)!^{h_i} \right).
    \end{equation}
\end{lemma}
\begin{proof}
First, note that if $\phi: V(H)\rightarrow V(H)$ is an automorphism of some fragment $H$ and $\hat{H}$ is an LLFr isomorphic to $H$ then permuting the labels in $\hat{H}$ according to $\phi$ yields the same LLFr. Hence, to derive $\gamma(H)$ we will count all the ways of labelling $V(H)$ to obtain an LLFr, and divide that number by $\aut(H)$. For a fragment $H$ the only multi-edges are those involved in $2$-cycles, which have multiplicity exactly two. By \eqref{eq:def_authe}, we can write $\authe(H)= \aut(H) 2^{a_1} 2^{a_2}$, and the {\sc RHS} of \eqref{eq:number_of_labelings} can be expressed as
\[
\frac{1}{\aut(H)} \left(\prod_{k\geq 1} a_k! \right)
\left( 2^{a_2} \prod_{k\geq 3} (2k)^{a_k} \right)
\left(
\prod_{i\geq 2} (i-2)!^{h^r_i} \right) \left( \prod_{i\geq 1} (i-1)!^{h_i} \right).
\]
We argue that the number of ways we can label $H$ to obtain an LLFr corresponds to the product of the parentheses above. Such labelling is uniquely given by (1) an ordering of the cycles in $H$ that is non-decreasing with respect to their lengths, (2) a cyclic ordering of the vertices inside of each cycle, and (3) an arbitrary ordering of each vertex's children.
There are $\prod_{k\geq 1} a_k!$ ways of achieving (1). Given a $k$-cycle, there are exactly $2$ cyclic orderings of its vertices when $k=2$, and $2k$ when $k\geq 3$. Thus, there are $2^{a_2} \prod_{k\geq 3} (2k)^{a_k}$ ways of achieving (2).
Finally, let us consider (3). Given a vertex of degree $i$ lying on a cycle, there are exactly $(i-2)!$ ways of ordering its children. Similarly, there are $(i-1)!$ ways of ordering the children of a vertex of degree $i$ not lying on a cycle. Hence, there are $\left(\prod_{i\geq 2} (i-2)!^{h^r_i} \right) \left( \prod_{i\geq 1} (i-1)!^{h_i} \right)$
ways of achieving (3). This shows the result. 

\end{proof}

We can now show that $p^*$ is a probability distribution.

\begin{theorem}
\label{thm:proba_dist}
    Assume $\nu<1$. Then $\sum_H p^*(H) = 1$, where $H$ ranges over all finite unlabeled fragments.
\end{theorem}
\begin{proof}
    Fix $\bm{a}=(a_k)_{k\geq 1}$ a sequence of non-negative integers  and suppose $A = \sum_{k\geq 1} k a_k<\infty$. 

    Let $\FF^{\bm{a}}$ be the set of all unlabeled fragments containing exactly $a_k$ $k$-cycles for each $k\geq 1$. Then, by \Cref{cor:acyclic}, to prove the statement it is enough to show that
    \begin{equation}
    \label{eq:frag_sum_aux2}
    \sum_{H\in \FF^{\bm{a}}} p^*(H) = \sqrt{1-\nu} \, \prod_{k\geq 1} \frac{(\nu^k/2k)^{a_k}}{a_k!}.
    \end{equation}
    Similarly, let $\FF^{\bm{a}}_\text{lex}$ be the set of all LLFr containing exactly $a_k$ $k$-cycles for each $k\geq 1$. By \Cref{lemma:numb_iso_LLFr}, we can write
    \begin{equation}
    \label{eq:using_LLFr}
\sum_{H\in \FF^{\bm{a}}} p^*(H) =\sum_{H\in \FF^{\bm{a}}_{\text{lex}}}
    \frac{p^*(H)}{\gamma(H)},  
    \end{equation}
    where $\gamma(H)$ is given in \eqref{eq:number_of_labelings}.
    Given $H\in \FF^{\bm{a}}_{\text{lex}}$ and $i\geq 1$, we write $h^r_i$ for the number of vertices of degree $i$ lying in some cycle of $H$, and $h_i$
    for the number of vertices of degree $i$ in $H$ that are not in a cycle. Since $\sum_{i\geq 2} h_i^r= \sum_{k\geq 1} k a_k$, it follows that,
\begin{align}
\nonumber
\frac{p^*(H)}{\gamma(H)} \, & = \,
\sqrt{1-\nu} \prod_{k\geq 1} \frac{1}{ a_k! (2k)^{a_k}
} \prod_{i\geq 2}  \left(\frac{\lambda_i i (i-1)}{\mmnt_1} \right)^{h^r_i} 
\prod_{i\geq 1}  \left(\frac{\lambda_i i}{\mmnt_1} \right)^{h_i} \\
\label{eq:frag_sum_aux1}
& = \, \sqrt{1-\nu} \prod_{k\geq 1} \frac{\nu^{k a_k}}{ a_k! (2k)^{a_k}
} \prod_{i\geq 2}  \left(\frac{\lambda_i i (i-1)}{\mmnt_2} \right)^{h^r_i} 
\prod_{i\geq 1}  \left(\frac{\lambda_i i}{\mmnt_1} \right)^{h_i}.
\end{align}

Let $v_1,\dots, v_A$
be the vertices belonging to the cycles in $H$, ordered in lexicographical order. 
We define the LLFo $F_H$, as the one containing $A$ trees, corresponding to the ones growing out of $v_1,\dots, v_A$ in that order. Observe that the map $H\mapsto F_H$
is a bijection between $\FF^{\bm{a}}_\text{lex}$ and the set $\TT^A_\text{lex}$ of LLFo consisting of $A$ components. See \Cref{fig:bijection} for an example.  \par
Consider the following distributions  $D$ and $D_r$ over non-negative integers:
\begin{align*}
    \Pr(D_r= i - 2) &=\frac{i(i-1)\lambda_i}{\mmnt_2},\; \text{ for all } i\geq 2,\\
    \Pr(D = i-1)&= \frac{i \lambda_i }{\mmnt_1},\quad\qquad\,  \text{ for all } i\geq 1.
\end{align*}

Using \eqref{eq:def_prob_trees}, we can rewrite \eqref{eq:frag_sum_aux1} as
\[
\frac{p^*(H)}{\gamma(H)} =\sqrt{1-\nu} \Big(\prod_{k\geq 1} \frac{\nu^{ka_k}}{ a_k! (2k)^{a_k}
}\Big) p_{F_H}^{\bm{D}}.
\]
From \eqref{eq:using_LLFr} and using the observation that $H\mapsto F_H$ is a bijection between $\FF^{\bm{a}}_{\text{lex}}$ and $\TT^A_{\text{lex}}$, we obtain 
\[
  \sum_{H\in \FF^{\bm{a}}} p^*(H) =  \sqrt{1-\nu} \prod_{k\geq 1} \frac{(\nu^k/2k)^{a_k}}{ a_k!}
    \sum_{F\in \TT^A_\text{lex}}
    p_{F}^{\bm{D}} =  \sqrt{1-\nu} \prod_{k\geq 1} \frac{(\nu^k/2k)^{a_k}}{ a_k!}.
\]
In the last equality we used that $ \sum_{F\in \TT^A_{\text{lex}}} p_{F}^{\bm{D}} = 1$ by \Cref{lem:branching_process}, since $\E{D}=\nu <1$.

\begin{figure}[tb]
    \centering
    \begin{subfigure}{0.40\textwidth}
    \includegraphics[width=\textwidth]{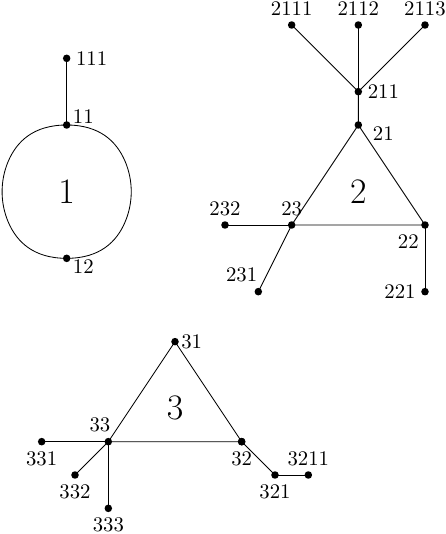}
    \caption{An LLFr.}
    \end{subfigure}
    \hfill
    \begin{subfigure}{0.44\textwidth}
    \includegraphics[width=\textwidth]{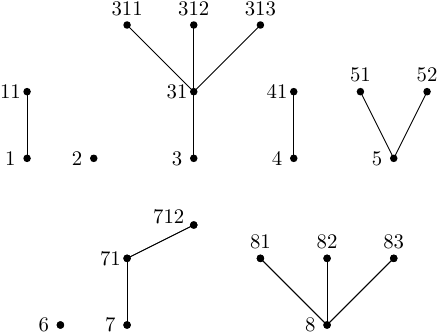}
    \caption{The corresponding LLFo.}
    \end{subfigure}
    \caption{Example of the map $H\mapsto F_H$.}
    \label{fig:bijection}
\end{figure} 

\end{proof}

\begin{corollary}
\label{cor:frag_integrable}
Assume $\nu<1$. Let $\FF$ be the class of unlabeled fragments and let $H\in \FF$. Then
the sequences
$(H\mapsto p^*_n(H))_{n\in \NN}$
and 
$(H\mapsto p_n(H))_{n\in \NN}$
of real maps over $\FF$
are tight. In particular, for all sequences $\omega_n $ tending to infinity as $n\to \infty$, $\Pr(|\Frag^*_n|\geq \omega_n )= o(1)$ and $\Pr(|\Frag_n|\geq \omega_n )=o(1)$.
\end{corollary}
\begin{proof}
The last part of the statement follows from the definition of tight sequences. The fact that
$(H\mapsto p^*_n(H))_{n\in \NN}$ is tight follows from \Cref{thm:proba_dist}. To see that $(H \mapsto p_n(H))_{n\in \NN}$ is tight as well, note that by definition
\begin{equation}
    p_n(H)\leq \frac{\Pr(\Frag^*_n \simeq H)}{\Pr( \Conf_n(\bd_n) \text{ is simple})}
\end{equation}
and
$\Pr( \Conf_n(\bd_n) \text{ is simple})\geq e^{-\nu/2 - \nu^4/4} - o(1)>0$. Thus $p_n(H)\leq Cp^*_n(H)$ and since the latter sequence is tight, so is the former one.
\end{proof}

%% file: sections/subcritical.tex
\section{First part of \texorpdfstring{\Cref{thm:main}}{Theorem 1.3}: A finite union of intervals}
\label{sec:subcritical}

In this section we show that $\overline{L(\bd)}$, the closure of the limit probabilities, is a union of closed intervals. We postpone the supercritical and critical cases $\nu\geq 1$ for the next section, and focus on the subcritical case $0<\nu<1$. The key point is that, in the subcritical regime, the $\fo$ properties of $\Gcal_n(\bd_n)$ are determined w.h.p. by its fragment $\Frag_n$. 
This is (implicitly) stated in \cite[Lemma 3.12]{lynchConvergenceLawRandom2005}. However, the results in \cite{lynchConvergenceLawRandom2005,lynchConvergenceLawRandom2003, } contain slight inaccuracies, that will be discussed in \Cref{sec:convergence_law}. 

\begin{theorem}[Zero-one Law for {\sc FO} in $\Gcal(\bd)=(\Gcal_n(\bd_n))_{n\in \mathbb{N}}$]
\label{thm:fragment_zeroone}
Suppose that $\nu<1$. Let $H\in \FF$ be an unlabelled fragment, and $\varphi\in \fog$ be a sentence. Then
\[
\Ln \Pr(\Gcal_n(\bd_n) \text{ satisfies }   \varphi \, \mid \, \Frag_n \simeq H )  \in \{ 0,1\}. 
\]
\end{theorem}

Recall that when $\nu<1$ and $H\in \FF$,
$p_{n}(H)=\Pr(\Frag_n\simeq H)$,
and $p(H)=\Ln p_n(H)$. We now prove that $L(\bd)$ equals the set of partial sums of fragment probabilities.

\begin{theorem} Assume $\nu<1$. Then
\label{thm:partial_sums}
\begin{equation}\label{eq:partial_sums}
\overline{L(\bd)}= \left\{ \sum_{H\in \Ucal} p_H \mmid
\Ucal \subseteq \FF
\right\}.
\end{equation}
\end{theorem}
\begin{proof}
Let $S(\bd)$ be the set of partial sums in the {\sc RHS} of \eqref{eq:partial_sums}.

We first show that $\overline{L(\bd)}\subseteq S(\bd)$.
It is a known fact (see e.g.~\cite{kakeya, nymann2000paper})
that $S(\bd)$ is closed and has no isolated points. Thus, $\overline{S(\bd)}=S(\bd)$, and it suffices to show
$L(\bd) \subseteq S(\bd)$.
Let $\phi\in \fog$ be a sentence.
For each $H\in \FF$, define 
\begin{align*}
p_n(\phi,H)&=
\Pr( \Gcal_n(\bd_n) \text{ satisfies }   \phi,\,
\Frag_n\simeq H)\\
&=
\Pr( \Gcal_n(\bd_n) \text{ satisfies }   \phi  \mid
\Frag_n\simeq H)\, p_n(H).
\end{align*}
Define
$
p(\phi)= \Ln
\sum_{H\in \FF} 
p_{n}(\phi, H)
$.
As $p_{n}(\phi, H)\leq 
p_n(H)$,
the sequence 
of real maps over $\FF$, 
$(H\mapsto p_{n}(\phi,H))_{n\in \NN}$ is tight, and the sum and limit in the definition of $p(\phi)$ may be exchanged.
By \Cref{thm:fragment_zeroone}, we know that
$\lim_{n\to \infty} \Pr( \Gcal_n(\bd_n) \text{ satisfies }   \phi  \mid
\Frag_n\simeq H)\in \{0,1\}$. Let $\Ucal=\Ucal_\phi\subseteq \FF$ be the set of fragments for which this limit is $1$. We conclude
\[
p(\phi)= 
\sum_{H\in \FF}  \Ln p_{n}(\phi, H) = \sum_{H\in \Ucal} p(H) \in S(\bd) .
\]

\par
We now show that $\overline{L(\bd)}\supseteq S(\bd)$.
Let $\Ucal \subseteq \FF$ be an arbitrary family of fragments.
We give a sequence of $\fog$ sentences $\phi_{k}(\Ucal)$
satisfying $\lim_{k\to \infty} p(\phi_{k}(\Ucal)) =
\sum_{H\in \Ucal} p(H)$. For each $H\in \FF$ and $k\in \NN$, let $\phi_{k}(H)\in \fog$ be the sentence stating that the graph
$G$ contains an isolated copy of $H$, and that no $k$-tuple of vertices outside this copy induce a cycle. 
Suppose that $\Ucal$ is infinite. 
Let $(\Ucal_i)_{i\in \NN}$ be a monotonically increasing chain of finite sets $\Ucal_i\subset \Ucal$ satisfying $\bigcup_{i\in \NN}
\Ucal_i = \Ucal$. Define
$\phi_{k}(\Ucal)=\bigvee_{H\in \Ucal_k} \phi_{k}(H)$.
The union of disjoint events $\left(\bigvee_{H\in \Ucal_k} \Frag_n\simeq H\right)$ implies
$\Gcal_n(\bd_n)\text{ satisfies }   \phi_{k}(\Ucal)$.
Let $A_{n,k}$ be the event that $\Gcal_n(\bd_n)$
contains a cycle of length larger than $k$.
Then, 
$(\Gcal_n(\bd_n)\text{ satisfies }   \phi_{k}(\Ucal))\wedge (\neg A_{n,k})$
implies 
$\left(\bigvee_{H\in \Ucal_k} \Frag_n\simeq H\right)$ as well.
Thus,
\begin{align}\label{eq:error}
\lvert 
p_n(\phi_{k}(\Ucal))-
\sum_{H\in \Ucal_k}
p_n(H) \rvert \leq \Pr(A_{n,k}).
\end{align}
By \Cref{lem:cycle_count_conf} and Markov inequality, 
$\lim_{k\to\infty} \Ln \Pr(A_{n,k})=0$. 

Taking limits in both sides of \eqref{eq:error}, first with respect to $n\to \infty$ and then to $k\to \infty$, we obtain,
\[
\lim_{k\to \infty} 
\left(p(\phi_{k}(\Ucal))-
\sum_{H\in \Ucal_k}
p(H)\right) = 0.\]
By the definition of infinite sum this proves that 
\[\sum_{H\in \Ucal}
p(H) = \lim_{k\to \infty} p(\phi_{k}(\Ucal)) \in \overline{L(\bd)}. \]

If $\Ucal$ is finite, then the proof follows from a simpler argument, by defining
$\phi_{k}(\Ucal) = \bigvee_{H\in\Ucal} \phi_{k}(H)$.
\end{proof}

The desired results about $\overline{L(\bd)}$
follow from analysing the set of fragment probabilities and using
Kakeya's Criterion (\Cref{ch:binom_lims.kakeya}). 
A technical difficulty that arises in the proof is that fragment probabilities depend on many more features of $\bd$ other than the parameter $\nu$. In order to circumvent this issue, we use the following lemma. 

\begin{lemma}
\label{lem:fragment_cycles}
    Suppose that $\nu<1$. Define $Q= Q(\nu) = \sqrt{1-\nu}\cdot e^{\nu/2 + \nu^2/4}$.  Let $\bm{a}=(a_n)_{n\geq 3}$ be a sequence of natural numbers $a_n\in \NN$
    with $\sum_{n\geq 3} a_n<\infty$. Consider
    \[\FF_{\bm{a}}=\{
    H\in \FF \mid H \text{ contains exactly $a_i$ $i$-cycles for each $i\geq 3$}
    \}.\]
    Then
    \begin{equation}
    \label{eqn:fragment_cycles}
    \sum_{H\in \FF_{\bm{a}}} p(H) =Q
    \prod_{i\geq 3} \frac{(\nu^i/2i)^{a_i}}{a_i!}.
    \end{equation}
    In particular, $p(H)$ is maximized when $H$ is the empty fragment. 
\end{lemma}
\begin{proof}
    Let $B_n$ be the event that $\Gcal_n(\bd_n)$ contains exactly $a_i$ $i$-cycles for each $i\geq 3$. By \Cref{cor:acyclic} it holds that
    \[
    \Pr(B_n) = Q
    \prod_{i\geq 3} \frac{(\nu^i/2i)^{a_i}}{a_i!} + o(1). 
    \]
    For each $H\in \FF$, let 
    \[
    q_n(H) = \Pr(B_n \mid \Frag_n\simeq H) \Pr(\Frag_n\simeq H).
    \]
    By the law of total probability $\Pr(B_n)= \sum_{H\in \FF} q_n(H)$. Moreover,
    observe that $q_n(H)\leq p_n(H)$ for all $H\in \FF$, so the sequence of maps
    $(H\mapsto q_n(H))_{n\in \NN}$ is tight. This way
    \begin{align}\label{eq:limit_qnH}
    \Ln \Pr(B_n) = \sum_{H\in \FF} \Ln q_n(H).
    \end{align}
    By \Cref{thm:complex}, we know that w.h.p. all cycles in $\Gcal_n(\bd_n)$ lie in $\Frag_n$. 
    
    This implies that $q_n(H)=p(H) + o(1)$ if $H\in \FF_{\bm{a}}$
    and $q_n(H)=o(1)$ otherwise. Using this in \eqref{eq:limit_qnH}, we obtain
    \eqref{eqn:fragment_cycles}. \par
    It remains to show that
    $p(H)$ is maximized at the empty fragment.
    Let $H\in \FF$ be non-empty, and let $\bm{a}=(a_i)_{i\geq 3}$ 
    be the sequence where $a_i$ is the number of $i$-cycles in $H$ for each
    $i\geq 3$. By \Cref{eqn:fragment_cycles}, 
    \[
    p(H)\leq
    Q \prod_{i\geq 3} \frac{(\nu^i/2i)^{a_i}}{a_i!}.\]
    However, as $\nu<1$, the expression on the right is at most
    $Q$, which is the probability
    of the empty fragment. This completes the proof. 
    \end{proof}

For the remainder of this subsection, we number the fragments in 
$\FF$ as $H_1, H_2,\dots$ in such a way that $p(H_i)\geq p(H_j)$ for all $i<j$. For convenience we define $p_i=p(H_i)$.
For each $i>1$, let
$k=k(i)$ be the number satisfying
\begin{equation}\label{eq:sandwich_pi}
\frac{Q \nu^k}{2k} \geq p_i > \frac{Q \nu^{k+1}}{2(k+1)},    
\end{equation}
where $Q = Q(\nu) = \sqrt{1-\nu}\cdot e^{\nu/2 + \nu^2/4}$ as in last lemma.
We impose the condition $i>1$, because $H_1$ corresponds to the empty fragment and $p_1=Q$ by \Cref{cor:acyclic}, so $k(1)$ is not well-defined. Observe that
\Cref{lem:fragment_cycles} implies $k(i)\geq 3$ for all $i>1$.
Finally, the probabilities $p_i$ are 
non-increasing by definition and have limit zero, so $k(i)$ is non-decreasing and tends to infinity with $i\to \infty$. \par

\begin{lemma}
\label{lem:finite_gaps}
Assume $\nu<1$. Then $\overline{L(\bd)}$ is a finite union of 
intervals in $[0,1]$.
\end{lemma}
\begin{proof}
Let $i_0$ be the smallest index $i>1$ for which
$
\sum_{j=3}^{k(i)-2} (1/j) \geq 4/\nu 
$, which exists as the harmonic series diverges.
We prove that 
$p_i\leq \sum_{j>i} p_j$ for any $i\geq i_0$. By Kakeya's Criterion, this implies the result. Let $i>i_0$, $k=k(i)$.
For each $3\leq \ell \leq \lfloor\frac{k+1}{2}\rfloor$, 
let $\FF_\ell$ be the set of unlabeled fragments containing an $\ell$-cycle, 
a $(k-\ell + 1)$-cycle, and no other cycle. 
By \Cref{lem:fragment_cycles}, it holds that 
\begin{equation}\label{eq:two_cycles}
    \sum_{H\in \FF_\ell} p(H)=
    \begin{cases}
       Q'/2 &\text{if $k$ is odd and } \ell=\frac{k+1}{2},\\
       Q'  &\text{otherwise},

    \end{cases}
\end{equation}
where $Q'=\frac{Q \nu^{k+1}}{4\ell(k-\ell+1)}$.
For any $\ell$, this sum is at most $\frac{Q \nu^{k+1}}{2(k+1)}$, which is at most $p_i$ by \eqref{eq:sandwich_pi}. In particular, $p(H)<p_i$ for all $H\in \FF_\ell$ and
\[\bigcup_{\ell=3}^{\lfloor\frac{k+1}{2}\rfloor}\FF_\ell \subset
\{ H_j \mid  j>i \}.\]

By the choice of $i_0$ and $i\geq i_0$, and by \eqref{eq:sandwich_pi}, we obtained the desired condition
\begin{align*}
    \sum_{j>i} p_j \geq &
    \sum_{\ell=3}^{\lfloor\frac{k+1}{2}\rfloor} \sum_{H\in\FF_\ell} p(H) \\
    &= 
    \frac{Q\,\nu^{k+1}}{8}
    \sum_{\ell=3}^{k-2} 
    \frac{1}{\ell(k-\ell+1)} \\
    &\geq
    \frac{Q\,\nu^{k+1}}{8k}\sum_{\ell=3}^{k-2} \frac{1}{\ell}\\ &\geq
    \frac{Q\, \nu^{k}}{2k}\geq p_i.\qedhere
    \end{align*}

\end{proof}

\section{Second part of \texorpdfstring{\Cref{thm:main}: Phase transition at \texorpdfstring{$\nu_0$}{v0}}.}

\label{sec:transition}

We recall that $\nu_0$ is defined as the unique root in $[0,1]$ of
\begin{equation}
    \sqrt{1-\nu}\cdot e^{\frac{\nu}{2} + \frac{\nu^2}{4}}=1/2.
\end{equation}

\begin{lemma}
\label{lem:transition}
The following hold.
\begin{enumerate}
    \item[(1)] if $0< \nu <\nu_0$, then $\overline{L(\bd)}$ has at least one gap, and
    \item[(2)] if $\nu\geq \nu_0$, then $\overline{L(\bd)}=[0,1]$. 
\end{enumerate}
\end{lemma}

\begin{proof}
The case $\nu\geq 1$ can be proven exactly as in \cite[Section 3.1]{larrauriLimitingProbabilitiesFirst2022}, using our results about the distribution of small cycles in $\Gcal_n(\bd_n)$ described in \Cref{sec:cycles}. We thus assume that $\nu<1$.

As in the previous subsection, let $H_1,H_2,\dots$
be an enumeration of the class of fragments $\FF$ satisfying 
$p(H_1)\geq p(H_2) \geq \dots$, and let $p_i= p(H_i)$ for all $i\geq 1$. 
By Kakeya's Criterion, $\overline{L(\bd)}=[0,1]$ if and only if
\begin{equation}
    \label{sec:subcritical.eq:kakeya}
   p_i\leq \sum_{j>i} p_j,
\end{equation}
for all $i\geq 1$.

We first show (1).
Recall that $\nu_0$ is defined as the only solution to $Q(\nu_0)=1/2$, which lies in the interval
$[0,1]$.
As $Q(\nu)$ is monotonically decreasing in $[0,1]$ (see~\eqref{eq:alter_acyclic}) and $0<\nu<\nu_0$, 
it holds that $Q(\nu)> 1/2$. Recall that $H_1$
corresponds to the empty fragment. By~\Cref{lem:fragment_cycles}, 
$$
p_1= Q> 1/2 > 1-Q = \sum_{j>1} p_j.
$$ 
and \eqref{sec:subcritical.eq:kakeya} does not hold for $i=1$, which implies that $\overline{L(\bd)}$ contains at least one gap.
\par
Now we proceed to show (2). In this case, $\nu_0\leq \nu < 1$, and \eqref{sec:subcritical.eq:kakeya} holds for $i=1$, because $Q\leq 1/2$. We show that \eqref{sec:subcritical.eq:kakeya}  holds for $i>1$ as well. Fix $i>1$ and let $k=k(i)$.
For all $\ell\geq 3$, we define $\FF_\ell$ as the set of unlabeled fragments containing an $\ell$-cycle and no other cycles. By~\Cref{lem:fragment_cycles},
$\sum_{H\in \FF_\ell} p_H=Q\nu^\ell/(2\ell)$. We have
\begin{equation}\label{eq:bounds_on_pj}
\sum_{j>i} p_j \geq 
\sum_{\ell > k} \sum_{H\in \FF_\ell}
p_H = Q \sum_{\ell> k} \frac{\nu^\ell}{2\ell}\geq 
\frac{Q\nu^{k}}{2k}\sum_{m \geq 1} \left(
\frac{\nu k}{k+1} 
\right)^m,
\end{equation}
where the last inequality follows from 
the fact that if $a_\ell=\nu^\ell/(2\ell)$, then $a_{\ell+1}\geq \frac{\nu k}{k+1}\,a_\ell$ for all $\ell\geq k$. Since $k(i)\geq 3$ for all $i> 1$, then $\frac{k}{k+1}\geq 3/4$. Note that $\nu_0 \geq 3/4$, so the {\sc LHS} of \eqref{eq:bounds_on_pj}  can be bounded as 
\[
\sum_{j>i} p_j \geq 
\frac{Q\nu^{k}}{2k}\sum_{m \geq 1} (3/4)^{2m} = \frac{9}{7}\cdot \frac{Q\nu^k}{2k}>\frac{Q\nu^k}{2k}\geq p_i,
\]
where we used \eqref{eq:sandwich_pi} in the last step. The criterion implies that $\overline{L(\bd)}=[0,1]$.

\end{proof}

%% file: sections/convergence_law.tex
\section{Remarks About the Convergence Law}\label{sec:convergence_law}

In this section we discuss the convergence law studied by Lynch in~\cite{lynchConvergenceLawRandom2005,lynchConvergenceLawRandom2003a}. His main result states that, under some conditions on the asymptotic degree sequence $\bd$, the limit of $\Pr(\Gcal_n(\bd_n)\text{ satisfies }   \phi)$ exists for any $\fo$ sentence $\phi$. We note that Lynch's requirements on $\bd$ are non-comparable with~\Cref{assump:main}.

\begin{definition}
    We call an asymptotic degree sequence $\bd$ \textit{smooth} if it satisfies conditions (i), (ii) and (iv) from~\Cref{assump:main}, as well as the following weakening of (iii):
    \[
    \lim_{n\to \infty} 
    \E{D_n} \text{ exists, is bounded and equals }  \E{D}.
    \]
\end{definition}

Both papers~\cite{lynchConvergenceLawRandom2005,lynchConvergenceLawRandom2003a} deal with smooth degree sequences. However, there is no condition on the convergence of $\E{D_n^2}$ to a finite quantity. Instead, this is replaced by a bound on the maximum degree: the existence of a cutoff function $\omega(n)$ satisfying $\Delta(n)\leq \omega(n)$. In~\cite{lynchConvergenceLawRandom2003a}, $\omega(n)=n^\alpha$, where $\alpha<1/4$, while in~\cite{lynchConvergenceLawRandom2005},  $\omega(n)$ was sub-polynomial (that is, $\omega(n)=o(n^\alpha)$ for all $\alpha>0$). 
Observe that neither cutoff is enough to guarantee that $\E{D_n^2}$ converges to a finite quantity (in fact, no diverging cutoff function is enough). This is relevant because of next result.

\begin{lemma}
 Let $\bd=\bd_n$ be a smooth asymptotic degree sequence with $\E{D^2_n}\to \infty$ as $n\to \infty$. Then $\Conf_n(\bd_n)$ a.a.s. contains a loop. 
\end{lemma}
\begin{proof}[Sketch of the proof]
Let $X_n$ count the number of loops in $\Conf_n(\bd_n)$. Then 
\[
\E{X_n}= \frac{1}{2}\sum_{v\in [n]} \frac{d_v (d_v-1)}{2m_n -1}=\frac{1}{2}\cdot\frac{\mmnt_{n,2}}{\mmnt_{n,1} - 1/n}.
\]
Using that $\mmnt_{n,2}$ diverges and $\bd$ is smooth we get that $\E{X_n}\to \infty$. The result follows from proving that $\Var{X_n}= o(\E{X_n^2})$ and using the second moment method. 
\end{proof}

The approach followed in~\cite{lynchConvergenceLawRandom2005,lynchConvergenceLawRandom2003a} consists of proving a {\sc FO} 
convergence law for the multigraph
$\Conf_n(\bd_n)$, and transferring this result to $\Gcal_n(\bd_n)$ afterwards by conditioning $\Conf_n(\bd_n)$ to the event of being simple. In order to make this work, it was taken for granted that the probability that $\Conf_n(\bd_n)$ is simple was bounded away from zero. However, because of last lemma, this is not true unless the second moment $\E{D^2_n}$ is bounded. \par

However, we claim that a $\Gcal(\bd)$ satisfies a {\sc FO} convergence law whenever $\bd$ follows~\Cref{assump:main}. Our aim is not to give a full proof of this statement, but we sketch it in the rest of the section. 

In~\cite{lynchsparse1992} Lynch showed that the binomial graph $\Gcal_n(p_n)$ satisfies a {\sc FO} convergence law when $p_n\sim c/n$ for any real constant $c>0$. Informally, this follows from three facts about $\Gcal_n(c/n)$. The \textit{$r$-core}
$\Core_r(G)$ of a graph $G$ is the graph induced by the $r$-neighborhood of all its cycles of size at most $2r+1$, and define $\Core_{n,r}$ as the $r$-core of $\Gcal_n(c/n)$. Then for any fixed $r>0$ the following hold:
\begin{enumerate}
    \item[(I)] w.h.p. any rooted tree of height at most $r$ appears as the $r$-neighbourhood of some vertex in $\Gcal_n(c/n)$ more than $K$ times, for any fixed integer $K>0$,
    \item[(II)] w.h.p. $\Core_{n,k}$ is a disjoint union of unicyclic graphs (i.e., a fragment),
    \item[(III)] and $\Core_{n,r}$ has a well-defined asymptotic distribution.
\end{enumerate}

We claim that the same three facts hold true in the multigraph $\Conf_n(\bd_n)$, with small changes. Let $\Core_{n,r}^*$ denote the $r$-core of $\Conf_n(\bd_n)$.

For Fact (I) we consider only rooted trees without \textit{forbidden degrees} according to $\bd$: those that do not have vertices of degree $k$, where $k$ satisfies the condition (iv) of \Cref{assump:main}. To see that this holds it is enough to observe that the $r$-neighbourhood of an arbitrary tuple of vertices in $\Conf_n(\bd_n)$ converges in distribution to a branching process with as many roots as vertices in the tuple, whose root offspring distribution is $D$ and general offspring distribution is $\widehat{D}$. Alternatively, one can use the second moment to show that there are many copies of any valid tree; this is precisely the part of the proof that requires condition (iv).

Fact (II) follows from a simple first-moment argument: a multigraph consisting of two small cycles that intersect has positive excess. The same holds true for two cycles joined by a short path. The proof is as~\Cref{thm:complex} but simpler: as we bound the size of the cycles and the paths, then there is only a finite number of forbidden configurations to be considered.

Fact (III) is the more convoluted one, we sketch the argument in what follows. The small-cycle distribution of $\Conf_n(\bd_n)$ converges to a vector of independent Poisson random variables, as shown in~\Cref{thm:cycle_distribution}. Consider an arbitrary disjoint union of cycles $H$, each of size at most $2r+1$. Let $A_{H}=A_{n,H}$ be the event that the union of cycles of size at most $2r+1$ in $\Conf_n(\bd_n)$ is isomorphic to $H$, and let $\Conf_n^H(\bd)$ denote $\Conf_n(\bd_n)$ conditioned on that event. Now, let $v_1,\dots, v_\ell$ be the fixed vertices lying on the $H$-copy of $\Conf^H_n(\bd)$. Delete the edges of $H$ and denote by $F$ the $r$-neighbourhood of $v_1,\dots, v_\ell$ in the resulting multigraph. Then $F$ converges in distribution to the first $r$-generations of a multi-rooted branching process with $\ell$ roots, offspring distribution $\hat{D}$ and root offspring distribution $\tilde{D}$ given by 
$$
\Pr(\tilde{D}=i-2) = \frac{i(i-1)\Pr(D=i)}{\mmnt_2}.
$$ 
Indeed, for each vertex in $H$, which correspond to the roots of $F$, we delete two edges. 
This shows that $\Core_{n,r}^*$ converges in distribution to a random fragment where the cycle counts are given by appropriate Poisson distributions, and the trees that grow out of the
cycles follow the distribution given by the branching process described above, proving (III).
Compare this with the interpretation of the limit distribution of the fragment obtained in \Cref{thm:proba_dist}, and observe the similarities. A random fragment $H$ following that distribution is constructed as follows: First generate the set of cycles of $H$, letting the number of $k$-cycles in $H$ follow a Poisson random variable with parameter $\nu^k/2k$, independently for each $k\geq 1$. Afterwards, attach to each vertex lying on a cycle an independent copy of the branching process with offspring distribution $\hat{D}$ and root offspring distribution $\tilde{D}$. In the setting of \Cref{thm:proba_dist}, the generated fragment was guaranteed to be finite because $\nu<1$. Here we do not have this assumption, but instead we bound the maximum size of the generated cycles to $2r+1$, and only consider the first $r$ generations of each branching process. 
\par

Facts (I), (II), and (III) for $\Conf(\bd)$ show a $\fo$-convergence law in the configuration model in the same fashion as shown in~\cite{lynchsparse1992}. Let $\varphi$ be a $\fo$-sentence, let $k$ be its quantifier rank and let $r= (3^k-1)/2 $.
Let $\Omega_r$ be the set of unlabeled fragments consisting of cycles of size at most $2r+1$ with trees of height at most $r$ attached to them.
Using pebble games 
and Fact (I) one can show that
\begin{equation}\label{eq:01law}
\Ln \Pr(\Conf_n(\bd_n) \text{ satisfies } \varphi \, \mid \,
\Core_{n,r}^*\simeq H) \in \{0,1\},    
\end{equation}
for any $H\in \Omega_r$. Facts (II) and (III) show that $\Core_{n,r}^*$ converges in distribution to a random graph from $\Omega_r$. Hence,
\begin{align*}
\Ln &\Pr(\Conf_n(\bd_n) \text{ satisfies } \varphi) = \\ 
&
\sum_{H\in \Omega_r} \Ln \Pr(\Conf_n(\bd_n) \text{ satisfies } \varphi \, \mid \,
\Core_{n,r}^*\simeq H)\cdot \Pr(\Core_{n,r}^*\simeq H).
\end{align*}
In the last sum, each of the first factors converge to either zero or one by \eqref{eq:01law}, and the right factors converge to a fixed probability by (III). This shows that the probability $\Conf_n(\bd_n)$ satisfies $\varphi$ converges, as desired.

%% file: sections/appendix.tex
\appendix

\section{Proof of auxiliary lemmas}\label{app:1}

\begin{lemma}
\label{lem:aux_ineq}
    Let $\alpha_1, \alpha_2,\dots, \alpha_k, \beta_1,\beta_2,\dots, \beta_k$ be positive integers satisfying $\alpha_i\geq \beta_i$ for all $i\in [k]$. Define $\alpha= \sum_{i\in [k]} \alpha_i$, and $\beta= \sum_{i\in [k]} \beta_i$.
    Then 
\begin{equation}
\label{eq:inequality_aux}
\prod_{i\in [k]} 
\prod_{0\leq j < \beta_i} 
(\alpha_i - j)   \leq 
   \frac{(\alpha)_\beta}{\alpha^\beta}\Big( \prod_{ \beta - k + 1\leq j < \beta} \frac{\alpha}{\alpha-j}\Big) \prod_{i\in [k]} \alpha_i^{\beta_i}.
\end{equation}
\end{lemma}
\begin{proof}
    The proof is by induction on $\beta$ for each $k$ and $\alpha_1,\dots, \alpha_k$ fixed. For $\beta=k$ the result is trivial. Suppose now that $\beta>k$. Then, for some $t\in [k]$ it must be that $\beta_t-1 \geq (\beta - k)\alpha_t/\alpha$. This is because
    \[
    \sum_{i\in [k]} \beta_i - 1 =
    \beta - k =  \sum_{i\in [k]}  (\beta - k)\alpha_i/\alpha.
    \]
    In particular, this means that
    \begin{equation}
\label{eq:inequality_aux_aux}
    \frac{\alpha_t}{\alpha_t - \beta_t + 1} \geq \frac{\alpha}{\alpha - \beta + k}.
    \end{equation}
    Observe that our assumption $\beta>k$ implies $\beta_t>1$. Additionally, by the induction hypothesis
    \[
   \prod_{i\in [k]} 
    \prod_{0\leq j < \beta^\prime_i} 
    \frac{\alpha_i}{\alpha_i - j}   \geq 
    \prod_{0\leq j < \beta^\prime - k + 1} \frac{\alpha}{\alpha-j},
    \]
    where $\beta^\prime_i = \beta_i$ for $i\neq t$, 
    $\beta^\prime_t = \beta_t - 1$, and $\beta^\prime = \beta - 1$. Multiplying by \eqref{eq:inequality_aux_aux} yields 

\begin{equation}
\label{eq:inequality_aux2}
\prod_{i\in [k]} 
\prod_{0\leq j < \beta_i} 
\frac{\alpha_i}{\alpha_i - j}   \geq 
\prod_{0\leq j < \beta - k + 1} \frac{\alpha}{\alpha-j}.
\end{equation}
Rearranging we obtain
 \begin{equation}
\label{eq:inequality_aux3}
\prod_{i\in [k]} 
\prod_{0\leq j < \beta_i} 
(\alpha_i - j)   \leq 
\prod_{0\leq j < \beta - k + 1} \frac{\alpha-j}{\alpha} \prod_{i\in [k]} \alpha_i^{\beta_i}.
\end{equation}   
Multiplying and dividing by $(\alpha)_\beta$ on the right hand side of the previous equation yields the desired result.
\end{proof}

\begin{lemma}
\label{lem:aux_ineq2}
    Let $\Delta=\Delta_N$ be a function on $N$ satisfying $\Delta_N=o(\sqrt{N})$.
    There is a sequence $\xi_N$ tending to $0$ as $N\to \infty$ such that:
\begin{itemize}
    \item[(i)] for all $0\leq a < \Delta$,
 \[
    \frac{N^a (N-a)!}{N!} \leq e^{a \xi_N }.
    \]
    \item[(ii)] for all $\Delta \leq a < N$,
    \[ 
    \frac{N^{\Delta}(N-a)! }{(N + \Delta - a)!} \leq e^{a \xi_N }
. \] 
\end{itemize}
\end{lemma}

\begin{proof}
    We begin with the proof of (ii), so suppose that  $\Delta \leq a < N$.
    By Stirling's approximation, we know that for all $k>0$ \[
    \sqrt{2\pi k}  \left(\frac{k}{e}\right)^k e^{\frac{1}{12k+1}} \leq
    k! \leq  \sqrt{2\pi k} \left(\frac{k}{e}\right)^k e^{\frac{1}{12k}}.\]
    Hence
    \[
    \frac{N^{\Delta}(N-a)! }{(N + \Delta - a)!}
    \leq 
    C D,
    \]
    where
    \begin{eqnarray} 
    C&=& 
    \left(1 + \frac{a- \Delta}{N + \Delta - a}\right)^{\Delta}
    \left(1 - \frac{\Delta}{N + \Delta - a}\right)^{N-a}
    e^{\Delta},\\
    D &=& \sqrt{\frac{N-a}{N + \Delta - a}} 
    e^{\frac{1}{12(N-a)} - \frac{1}{12(N-a + \Delta) + 1}}.
    \end{eqnarray}
    Clearly $D\leq e^{1/12}$ for all $\Delta \leq a < N$, so 
    $\ln(D)/a$ tends to zero uniformly with $N$.
    Now we need to prove the same for $\ln(C)/a$.
    We consider two cases. 
    First, suppose that $N-a\leq N^{2/3}$. Since $\Delta=o(\sqrt{N})$, we have for $N$ large enough
    \[
    C\leq e^{\Delta}(N/\Delta)^{\Delta} \leq e^{N^{2/3}}
    \]
    Since $N\sim a$, we have that $\ln(C)/a$ tends to zero.

    Otherwise, suppose that $N-a\geq N^{2/3}$. 
    Using the inequality $1+x\leq e^x$ for all $x\in \mathbb{R}$, we get that
    \[
    C \leq 
    \exp\left[ 
    \Delta\frac{
    (a - \Delta) - (N + a) + ( N - a + \Delta) }{
    N - a + \Delta}
    \right] 
    = \exp\left[ 
    \frac{ \Delta a}{
    N - a + \Delta}
    \right].
    \]
    Thus, all $\Delta \leq a \leq \Delta \sqrt{N}$, 
    \[
    \frac{\ln(C)}{a} = 
    \frac{ \Delta }{
    N - a + \Delta}
    \leq 
    \frac{ \Delta }{
    N^{2/3}},         
    \]
    which tends to zero because
    $\Delta= o(\sqrt{N})$. \par
    
    Now let us show (i), so suppose that $0\leq a < \Delta$.
    It holds that
    \[
    \frac{N^a (N-a)!}{N!} \leq \left(\frac{N}{N-a}\right)^a 
    \leq e^{a^2/(N-a)}\leq e^{a \frac{\Delta}{N-\Delta}},
    \]
    where we have used that $a<\Delta < N$ in the last inequality.
    The function $\frac{\Delta}{N-\Delta}$ tends to zero with $N$ and depends only on
    $N$ and $\Delta$, as we wanted. 
    
\end{proof}